%% file: whirlcompress2.tex
\documentclass[12pt]{amsart}
\usepackage{hyperref}
\usepackage{amssymb}
\usepackage{epsfig}

\usepackage[margin=1.1in]{geometry}

\newtheorem*{statement}{Theorem}
\newtheorem{theorem}{Theorem}[section]
\newtheorem{thm}[theorem]{Theorem}
\newtheorem{lemma}[theorem]{Lemma}
\newtheorem{lem}[theorem]{Lemma}
\newtheorem{proposition}[theorem]{Proposition}
\newtheorem{prop}[theorem]{Proposition}
\newtheorem{corollary}[theorem]{Corollary}
\newtheorem{cor}[theorem]{Corollary}
\newtheorem{conjecture}[theorem]{Conjecture}

\theoremstyle{definition}

\newtheorem{problem}[theorem]{Problem}

\newtheorem{question}[theorem]{Question}
\theoremstyle{remark}
\newtheorem{example}{Example}[section]
\newtheorem{remark}{Remark}[section]
\newcommand\remind[1]{{\bf **** #1 }}

\def\L{{\mathbb L}}
\def\a{{\mathbf a}}
\def\b{{\mathbf b}}
\def\c{{\mathbf c}}

\def\xx{{\mathbf x}}

\def\cc{{\mathbb C}}
\def\R{{\mathbb R}}
\def\C{\mathfrak C}
\def\h{\mathfrak h}
\def\X{\overline X}
\def\Y{\overline Y}
\def\Z{{\mathbb Z}}
\def\x{\overline x}
\def\r{\overline r}
\def\s{\overline s}
\def\e{\overline e}

\def\G{GL_n(\R((t)))}
\def\g{GL_n(\R[t,t^{-1}])}

\def\wt{{\rm wt}}

\def\Imm{\mathrm{Imm}}

\def\supp{{\rm supp}}
\def\LSym{{\rm LSym}}

\def\curl{{\rm rot}}

\begin{document}

\author{Thomas Lam and Pavlo Pylyavskyy}

\thanks{
T.L. was partially supported by NSF grants DMS-0600677, DMS-0652641
and DMS-0901111, and by a Sloan Fellowship.  P.P. was partially
supported by NSF grant DMS-0757165. Part of this work was completed
during a stay at MSRI} \email{tfylam@umich.edu}
\email{pavlo@umich.edu}

\title{Total positivity in loop groups I: whirls and curls}

\begin{abstract}
This is the first of a series of papers where we develop a theory of
total positivity for loop groups.  In this paper, we completely
describe the totally nonnegative part of the polynomial loop group
$GL_n(\R[t,t^{-1}])$, and for the formal loop group $GL_n(\R((t)))$
we describe the totally nonnegative points which are not totally
positive.  Furthermore, we make the connection with networks on the
cylinder.

Our approach involves the introduction of distinguished generators,
called whirls and curls, and we describe the commutation relations
amongst them. These matrices play the same role as the poles and
zeroes of the Edrei-Thoma theorem classifying totally positive
functions (corresponding to our case $n=1$). We give a solution to
the ``factorization problem'' using limits of ratios of minors. This
is in a similar spirit to the Berenstein-Fomin-Zelevinsky Chamber
Ansatz where ratios of minors are used.
A birational symmetric group action arising in the commutation relation
of curls appeared previously in Noumi-Yamada's study of discrete Painlev\'{e} dynamical systems 
and Berenstein-Kazhdan's study of geometric crystals.

\end{abstract}
\maketitle

\DeclareRobustCommand{\SkipTocEntry}[5]{} \tableofcontents
\tableofcontents

\section{Introduction}\label{sec:intro}

A matrix with real entries is {\it {totally nonnegative}} if all of
its minors are nonnegative.

\addtocontents{toc}{\SkipTocEntry}
\subsection{Total positivity in loop groups}\label{sec:TP}
Suppose $A(t)$ is a matrix with entries which are real polynomials,
or real power series.  When do we say that $A(t)$ is totally
nonnegative?  First associate to $A(t)$ an infinite periodic matrix
$X$, as in the following example:
\begin{eqnarray*}
\begin{array}{c}\left(
\begin{array}{cc}
1 + 9t^2 & 2 + 5t \\
-1 -2t -3t^2  & 8+3t-4t^2 \end{array} \right) =  \\
\left(\begin{array}{cc}1 & 2\\ -1 & 8\end{array}\right) + t
\left(\begin{array}{cc}0 & 5\\ -2 & 3\end{array}\right) + t^2
\left(\begin{array}{cc}9 & 0 \\ -3 & 4\end{array} \right)
\end{array} &\hspace{-20pt}\rightsquigarrow& \hspace{-10pt}\left(
\begin{array}{c|cc|cc|cc|c}
\ddots & \vdots & \vdots & \vdots  & \vdots  & \vdots&\vdots \\
\hline
\dots& 0 & 5& 9& 0&0& 0&\dots \\
\dots&-2&3&-3&4&0&0&\dots \\
\hline
\dots&1  & 2 & 0 & 5 & 9 & 0&\dots\\
\dots& -1 & 8 &-2 & 3 &-3&4 & \dots \\
\hline
\dots&0&0&1  & 2 & 0 & 5 & \dots\\
\dots&0&0& -1 & 8 &-2 & 3 & \dots \\
\hline & \vdots & \vdots & \vdots  & \vdots  & \vdots & \vdots &
\ddots\end{array} \right) \\
A(t) \hspace{100pt} && \hspace{80pt} X
\end{eqnarray*}
We declare that $A(t)$ is totally nonnegative if and only if $X$ is
totally nonnegative.  We use this to define and study the totally
nonnegative part of the loop groups $\g$ and $\G$.  Here $\R((t))$
denotes the field of formal Laurent series.  We let $\G_{\geq 0}$
denote the totally nonnegative part of $\G$.  Our main aim is to
unify and generalize two classical subjects: total positivity in
$GL_n(\R)$ and totally positive functions.

\addtocontents{toc}{\SkipTocEntry}
\subsection{Total positivity in $GL_n(\R)$}
The theory of totally positive matrices began in the 1930's in the
works of Schoenberg \cite{Sch} and Gantmacher-Krein \cite{GK} who
discovered that totally positive matrices had remarkable spectral
properties and a variation-diminishing property, cf. \cite{K}.

Let $e_i(a) \in GL_n(\R)$ (resp.~$f_i(a) \in GL_n(\R)$) be the
Chevalley generators, which differ from the identity matrix by a
single entry in the $i$-th row (resp.~column) equal to $a \in \R$
immediately above (resp.~below) the diagonal.  From our point of
view, the most important classical result is:
\begin{thm}[Loewner-Whitney Theorem {\cite{Lo,Wh}}]\label{T:LW}
The space of non-singular totally nonnegative matrices $GL_n(\R)_{\geq 0}$ is
the multiplicative semigroup generated by Chevalley generators $e_i(a)$, $f_i(a)$
with positive parameters, and positive diagonal matrices.
\end{thm}
Theorem \ref{T:LW} led Lusztig \cite{Lu1} to his ground-breaking
generalization of total positivity to reductive groups.  Lusztig
discovered deep connections between the theory of total positivity
and his own theory of canonical bases in quantum groups \cite{Lu2}.
In another direction, Fomin and Zelevinsky \cite{FZ1,FZ2} studied
the problem of parametrizing and testing for totally nonnegative
matrices. Their attempt to classify the ways to test whether a
matrix is totally nonnegative eventually led to the theory of {\it
cluster algebras} \cite{FZ3}.

Our first theorem (Theorem \ref{T:polygen}) establishes the analogue
of Theorem \ref{T:LW} for the totally nonnegative part $\g_{\geq 0}$
of the polynomial loop group, using the affine Chevalley generators.
Note that the polynomial loop group itself is {\it not} generated by
the torus and affine Chevalley generators with arbitrary parameters.
%

\addtocontents{toc}{\SkipTocEntry}
\subsection{Totally positive functions}
A formal power series $a(t) = 1 + a_1 t + a_2 t^2 + \cdots \in
\R[[t]]$ can be considered a $1 \times 1$ matrix.  We may then apply
the definition of total nonnegativity in $GL_1(\R((t)))$ of
subsection \ref{sec:TP} to define when a formal power series is
totally nonnegative.  Traditionally, formal power series $a(t)$
which are totally nonnegative are called {\it totally positive
functions}. The coefficients $\{a_1, a_2,\ldots \}$ are said to form
a {\it {Polya frequency sequence}}, see \cite{Br2}. Totally positive
functions were classified independently by Edrei and Thoma
\cite{E,Th}.

\begin{thm}[Edrei-Thoma theorem]\label{T:ET}
Every totally positive function $a(t)$ has a unique expression as
$$a(t) = e^{\gamma t}\frac{\prod_i (1+\alpha_i t)}{\prod_i
(1- \beta_i t)},$$ where $\alpha_i, \beta_i$ and $\gamma$ are
nonnegative parameters satisfying $\alpha_1 \geq \alpha_2 \geq
\ldots$, $\beta_1 \geq \beta_2 \geq \ldots$ and $\sum_i \alpha_i +
\sum_i \beta_i < \infty$.  In particular, totally positive functions
are meromorphic functions, holomorphic in a neighborhood of 0.
\end{thm}
Thoma \cite{Th} showed that the classification of totally positive
functions was equivalent to the classification of characters of the
infinite symmetric group $S_\infty$.  This connection was made more
robust when Vershik and Kerov \cite{VK} interpreted the zeroes and
poles in Theorem \ref{T:ET} as asymptotic frequencies occurring in
the representation theory of $S_\infty$.  No completely elementary
proof of Theorem \ref{T:ET} seems to be known. For example, the
original proofs of Edrei and Thoma use Nevanlinna theory from
complex analysis, while Okounkov's proofs \cite{O} rely on the
connection with asymptotic representation theory.

One of the main themes of our work is the parallel between Theorem
\ref{T:ET} and Theorem \ref{T:LW}: $(1+ \alpha t)$, $1/(1-\beta t)$,
and $e^{\gamma t}$ can be thought of as semigroup generators for
totally positive functions, when we also allow taking limits of
products. We begin by considering the analogues of these generators
for $n > 1$.

\addtocontents{toc}{\SkipTocEntry}
\subsection{Whirls and curls}
We introduce matrices $M(a_1,a_2,\ldots,a_n) \in \G$ called {\it
whirls}, and $N(b_1,b_2,\ldots,b_n) \in \G$, called {\it curls},
depending on $n$ real (usually nonnegative) parameters. For $n = 2$,
their infinite periodic representations look like
$${\footnotesize{
M(a_1, a_2) = \left(\begin{array}{ccccccc} \ddots & \vdots & \vdots & \vdots & \vdots &\vdots \\
\cdots& 1& a_1 & 0 & 0 & 0& \cdots \\ \cdots& 0& 1& a_2 & 0 & 0 &
 \cdots \\
 \cdots& 0&0 & 1& a_1& 0&\cdots \\
 \cdots& 0&0&0&1& a_2& \cdots \\
  \cdots& 0&0&0&0&1&  \cdots \\
 & \vdots & \vdots & \vdots & \vdots &\vdots &\ddots
\end{array} \right)  N(b_1,b_2) = \left(\begin{array}{ccccccc} \ddots & \vdots & \vdots & \vdots & \vdots &\vdots \\
\cdots& 1& b_1 & b_1 b_2 & b_1^2 b_2 & b_1^2b_2^2& \cdots \\ \cdots& 0& 1& b_2 & b_1 b_2 &
b_1 b_2^2 &
 \cdots \\
 \cdots& 0&0 & 1& b_1& b_1 b_2&\cdots \\
 \cdots& 0&0&0&1& b_2& \cdots \\
  \cdots& 0&0&0&0&1&  \cdots \\
 & \vdots & \vdots & \vdots & \vdots &\vdots &\ddots
\end{array} \right)}}
$$
Unlike Theorem \ref{T:ET}, our theory is not commutative when $n >
1$. We study whirls and curls in detail.  In Section \ref{sec:comm},
we describe the commutation relations for whirls and curls.  In
Section \ref{sec:semigroup}, we define the notion of infinite
products of whirls or curls, and show (see Theorems \ref{thm:semi},
\ref{thm:lsemi} and \ref{T:Chevcurl}) the following.
\begin{statement}[Structure of infinite whirls and curls]
Infinite products of whirls (or curls) form semigroups which are
closed under multiplication by Chevalley generators on one side.
\end{statement}

\addtocontents{toc}{\SkipTocEntry}
\subsection{The totally positive part $\G_{>0}$} If $X$ is an
infinite periodic matrix corresponding to $A(t) \in \G$, then every
sufficiently southwest entry of $X$ is necessarily equal to 0.  Thus
$X$ is never totally positive in the usual sense, which requires all
minors to be strictly positive.  We define $A \in \G_{\geq 0}$ to be
{\it totally positive} if it is totally nonnegative, and in
addition, all sufficiently northeast minors (see subsection
\ref{ssec:TP} for the precise definition) of the corresponding
infinite periodic matrix are strictly positive. We show (Theorem
\ref{T:notTP}):

\begin{statement}[Matrices of finite type]
The set $\G_{\geq 0} - \G_{>0}$ of totally nonnegative matrices in
the formal loop group which are not totally positive is a semigroup
generated by positive Chevalley generators, whirls, curls, shift
matrices (defined in Section \ref{sec:upper}), and diagonal
matrices.
\end{statement}

\addtocontents{toc}{\SkipTocEntry}
\subsection{Canonical form}
For simplicity, we restrict (using Theorem \ref{T:red}) to the
subsemigroup $U_{\geq 0} \subset \G_{\geq 0}$ consisting of matrices
$A(t)$ with upper triangular infinite periodic representations. In
Theorems \ref{thm:canon} and \ref{thm:regular}, we establish a
partial generalization of Theorem \ref{T:ET} to $n > 1$ (it is in
fact a rather precise generalization of the result of Aissen,
Schoenberg, and Whitney \cite{ASW}). We call a matrix $Y \in U_{\geq
0}$ {\it entire} if all $n^2$ matrix entries are entire functions.
The following results are our main theorems.

\begin{statement}[Canonical Form I]
Every $X \in U_{\geq 0}$ has a unique factorization as $X =
Z\exp(Y)W$, where $Z$ is a (possibly infinite) product of curls, $W$
is a (possibly infinite) product of whirls, and $Y$ is entire such
that $\exp(Y) \in U_{\geq 0}$.
\end{statement}

The ``limits of products'' $A$ and $B$ in the following theorem are
not necessarily single infinite products.

\begin{statement}[Canonical Form II]
Every matrix $\exp(Y) \in U_{\geq 0}$ with $Y$ entire, has a
factorization as $\exp(Y) = AVB$, where $A$ and $B$ are both limits
of products of Chevalley generators, and $V \in U_{\geq 0}$ is {\it
regular}.
\end{statement}
In \cite{LP}, we strengthen this result by showing that the matrices
$A, V, B$ in the above theorem are unique.   The notion of regular
totally nonnegative matrices is introduced and discussed in Section
\ref{sec:canon}. These results establish that every $X \in U_{\geq
0}$ has three ``components'': (a) a whirl and curl component, (b) a
component consisting of products of Chevalley generators, and (c) a
regular totally nonnegative matrix. We study (a) in detail here, but
leave (b) and (c) for subsequent papers \cite{LP,LP3}.

\addtocontents{toc}{\SkipTocEntry}
\subsection{From planar networks to cylindric networks}
A fundamental property of totally positive matrices is their
realizability by planar weighted networks, connecting total
positivity with combinatorics.  By the Lindstr\"om theorem \cite{Li}
and Theorem \ref{T:LW} (see also \cite{Br}) a matrix $X \in
GL_n(\R)$ is totally nonnegative if and only if it is ``realizable''
by a planar weighted directed acyclic network.  In Section
\ref{sec:networks}, we prove (Theorem \ref{thm:network}) an
analogous statement for loop groups: a matrix $X \in \g$ is totally
nonnegative if and only if it is ``realizable'' by a weighted
directed acyclic network on a {\it cylinder} (see for example Figure
\ref{fig:loo6}).

In the classical (planar) case, the minors of the matrix $X \in
GL_n(\R)$ are interpreted in terms of non-intersecting families of
paths.  Using the winding number of paths on a cylinder, we define a
notion of pairs of paths being ``uncrossed'' ({\it not} the same as
non-crossing).  The analogous interpretation (Theorem \ref{thm:L})
of minors of $X \in \g_{\geq 0}$ involves uncrossed families of
paths on the cylinder, and includes {\it some} paths which do
intersect.

The idea of using a chord on a cylinder to keep track of the winding
number, as it is done in this paper, appeared first in the work of
Gekhtman, Shapiro, and Vainshtein \cite{GSV}, which remained
unpublished for some time.


\addtocontents{toc}{\SkipTocEntry}
\subsection{Factorization problem}
In \cite{BFZ}, Berenstein, Fomin and Zelevinsky study the problem of
finding an expression for the parameters $t_1, t_2, \ldots, t_\ell
\in \R_{> 0}$ in terms of the matrix entries of $X =
e_{i_1}(t_1)e_{i_2}(t_2) \cdots e_{i_{\ell}}(t_\ell)$.  They solve
the problem by writing the parameters $t_i$ as ratios of minors of
the ``twisted matrix'' of $X$.  This inverse problem led to the study
of double wiring diagrams and double Bruhat cells \cite{FZ2}, and
later contributed to the discovery of cluster algebras \cite{FZ3}.

In Section \ref{sec:fact}, we pose and solve a similar question in
our setting.  For a matrix $X$ which is an infinite product of
curls, we identify a particular factorization into curls, called
the {\it {ASW factorization}}.  Roughly speaking, the ASW
factorization has curls ordered by radius of convergence.  We
express (Theorem \ref{thm:mrl} and Corollary \ref{C:mrl}) the
parameters of the curls in the ASW factorization as limits of ratios
of minors of $X$. Other factorizations of $X$ into curls are
obtained from the ASW factorization by the action of the infinite
symmetric group $S_{\infty}$.

\addtocontents{toc}{\SkipTocEntry}
\subsection{Loop symmetric functions}
One of the technical tools we use throughout the paper is a theory
of tableaux for a Hopf algebra we call {\it loop symmetric
functions}, denoted $\LSym$.  For $n = 1$, we obtain the usual
symmetric functions.  Roughly speaking, $\LSym$ generalizes usual
symmetric functions in the same way matrix multiplication
generalizes scalar multiplication.  The points of $\G_{\geq 0}$ are
in bijection with algebra homomorphisms $\phi: \LSym \to \R$ which
take nonnegative values on a particular spanning set of $\LSym$.  We
leave the detailed investigation of $\LSym$ for future work.  In the
present article we define $\LSym$ analogues of homogeneous and elementary
symmetric functions, tableaux, and Schur functions, and give a
Jacobi-Trudi formula (Theorem \ref{T:schur}).

\addtocontents{toc}{\SkipTocEntry}
\subsection{Curl commutation relations, birational $R$-matrix, and
discrete Painlev\'{e} systems} The commutation relations for curls
give rise to a birational action of the symmetric group on a
polynomial ring, for which $\LSym$ is the ring of invariants.  This
birational action was studied extensively by Noumi and Yamada
\cite{NY, Y} in the context of discrete Painlev\'{e} dynamical
systems (see also \cite{Ki}).  It also occurs as a birational
$R$-matrix in the Berenstein-Kazhdan \cite{BK} theory of geometric
crystals (see also \cite{Et}). The tropicalization of this
birational action is the combinatorial $R$-matrix of affine
crystals, studied in \cite{KKMMNN}.

We hope to clarify these unexpected connections in the future.

\addtocontents{toc}{\SkipTocEntry}
\subsection{Future directions and acknowledgements.}
Our work suggests many future directions.  For example:

{\it What asymptotic representation theory corresponds to total
nonnegativity of the formal loop group?} (see \cite{Th, VK,Ol,O})

{\it How does our work generalize to loop groups of other types?}
(see \cite{Lu1})

{\it Is there an ``asymptotic'' notion of a cluster algebra?} (see
\cite{FZ3})

We also give a list of precise problems, conjectures and questions
in Section \ref{sec:prob}.

\smallskip

We thank Alexei Borodin and Bernard Leclerc for discussing this work
with us.  We are grateful to Michael Shapiro for familiarizing us with some of the ideas in \cite{GSV}. We also thank Sergey Fomin for many helpful comments, and
for stimulating this project at its early stage.

\section{The totally nonnegative part of the loop group}
\subsection{Formal and polynomial loop groups}
An integer $n \geq 1$ is fixed throughout the paper.  If $i \in \Z$,
we write $\bar i$ for the image of $i$ in $\Z/n\Z$.  Occasionally,
$\bar i$ is treated as an element of $\Z$, in which case we pick the
representatives of $\Z/n\Z$ in $\{1,2,\ldots,n\}$.

Let $\G$ denote the {\it formal loop group}, consisting of $n \times
n$ matrices $A(t) = (a_{ij}(t))_{i,j = 1}^n$ whose entries are
formal Laurent series of the form $a_{ij}(t) =\sum_{k \geq
-N}^\infty b_k t^k$, for some real numbers $b_k \in \R$ and an
integer $N$, and such that $\det(A(t)) \in \R((t))$ is a non-zero
formal Laurent series. We let $\g \subset \G$ denote the polynomial
loop group, consisting of $n \times n$ matrices with Laurent
polynomial coefficients, such that the determinant is a non-zero
monomial.  We will allow ourselves to think of the rows and
columns of $A(t)$ to be labeled by $\Z/n\Z$, and if no confusion
arises we may write $a_{ij}(t)$ for $a_{\bar i \bar j}(t)$, where $i
, j \in \Z$.

To a matrix $A(t) = (a_{ij}(t)) \in \G$, we associate a
doubly-infinite, periodic, real matrix $X=(x_{i,j})_{i,j = -
\infty}^{\infty}$ satisfying $x_{i+n,j+n}=x_{i,j}$ for any $i,j$,
called the {\it unfolding} of $A(t)$, defined via the relation:
$$
a_{ij}(t) = \sum_{k=-\infty}^{\infty} x_{i,j+kn} t^{k}.
$$
We call $A(t)$ the {\it folding} of $X$, and write $A(t) = \X(t)$
for this relation.  Clearly, $\X(t)$ and $X$ determine each other
and furthermore we have $XY = Z$ if and only if $\X(t) \Y(t) =
\overline{Z}(t)$. We abuse notation by writing $X \in \G$ or $X \in
\g$ if the same is true for $\X(t)$.  If $X \in \G$, we also write
$\det(X)$ for $\det(\X(t))$.  We define the {\it support} of $X$ to
be the set $\supp(X) = \{(i,j) \in \Z^2 \mid x_{ij} \neq 0\} $.
\begin{example} \label{ex:1}
For $n=2$, an element of $\G$ and its unfolding are
$$
 \left(\begin{array}{cc}
\cosh(\sqrt{abt}) & \sqrt{a/bt} \sinh(\sqrt{abt}) \\
\sqrt{bt/a} \sinh(\sqrt{abt}) & \cosh(\sqrt{abt})
\end{array} \right)
\rightsquigarrow
\left(\begin{array}{ccccccc} \ddots & \vdots & \vdots & \vdots & \vdots &\vdots \\
\cdots& 1& a & \frac{ab}{2} & \frac{a^2b}{6} & \frac{a^2b^2}{24}& \cdots \\ \cdots& 0& 1& b & \frac{ab}{2} &
\frac{a b^2}{6} &
 \cdots \\
 \cdots& 0&0 & 1& a& \frac{ab}{2}&\cdots \\
 \cdots& 0&0&0&1& b& \cdots \\
  \cdots& 0&0&0&0&1&  \cdots \\
 & \vdots & \vdots & \vdots & \vdots &\vdots &\ddots
\end{array} \right)
.
$$
\end{example}
\begin{example} \label{ex:2}
For $n=3$, an element of $\g$ and its unfolding are
$$
 \left(\begin{array}{ccc}
3 & 1 & 2 t^{-1} \\
1+t & 2 & 1 \\
t & 0 & 1
\end{array} \right)
\rightsquigarrow
\left(\begin{array}{cccccccc} \ddots & \vdots & \vdots & \vdots & \vdots & \vdots &\vdots \\
\cdots& 3& 1 & 0 & 0 & 0& 0 &\cdots \\ \cdots& 1& 2& 1 & 1 &
0 & 0 &
 \cdots \\
 \cdots& 0&0 & 1& 1& 0& 0 &\cdots \\
 \cdots& 0&0&2&3& 1& 0 &\cdots \\
  \cdots& 0&0&0&1&2& 1 & \cdots \\
    \cdots& 0&0&0&0&0& 1 & \cdots \\
 & \vdots & \vdots & \vdots & \vdots & \vdots &\vdots &\ddots
\end{array} \right)
.
$$
\end{example}

For a real parameter $a \in \R$ and an integer $k$, we define
$e_k(a) = (x_{i,j})_{i,j=-\infty}^\infty \in \g$ to be the matrix
given by
$$
x_{i,j}= \begin{cases} 1 & \mbox{if $i = j$} \\
 a & \mbox{if $j = i+1$ and $\bar i = \bar k$} \\
 0 & \mbox{otherwise.} \end{cases}
$$
Similarly, define $f_k(a)  \in \g$ to be the transpose of $e_k(a)$.
We call the $e_j$-s and $f_j$-s {\it {Chevalley generators}}.

\subsection{Totally nonnegative matrices}\label{ssec:TP}
If $X \in \G$, and $I \subset \Z$ and $J \subset \Z$ are finite sets
of equal cardinality, we write $\Delta_{I,J}(X)$ for the minor of
$X$ obtained from the rows indexed by $I$ and columns indexed by
$J$.  We write $X_{I,J}$ to denote a submatrix, so that
$\det(X_{I,J}) = \Delta_{I,J}(X)$.

Let us say that $X \in \G$ is {\it {totally nonnegative}}, or TNN
for short, if every finite minor of $X$ is nonnegative.  We write
$\G_{\geq 0}$ for the set of totally nonnegative elements of $\G$.
Similarly, we define $\g_{\geq 0}$.  We say that $X \in \G_{\geq 0}$
is {\it totally positive} if there exists an integer $k$ such that
for every pair of subsets $I = \{i_1 < i_2 < \cdots < i_r\} \subset
\Z$ and $J = \{j_1 < j_2 < \cdots < j_r\} \subset \Z$ satisfying
$i_t \leq j_t + k$ for each $t \in [1,r]$, we have $\Delta_{I,J}(X)
> 0$. In other words, $X$ is totally positive if every sufficiently
northeast minor is strictly positive.  We denote the totally
positive part of $\G$ by $\G_{> 0}$. Note that $\G_{>0} \cap \g =
\emptyset$.
\begin{example}
The matrices in both Example \ref{ex:1} and Example \ref{ex:2} are totally nonnegative. The matrix in Example \ref{ex:1} can be shown to be totally positive.
\end{example}

\begin{lem}\label{lem:product}
The sets $\G_{\geq 0}$, $\G_{> 0}$ and $\g_{\geq 0}$ are semigroups.
\end{lem}
\begin{proof}
Follows immediately from the Cauchy-Binet formula which states that
\begin{equation}\label{E:CB}
\Delta_{I,J}(XY) = \sum_{K} \Delta_{I,K}(X) \Delta_{K,J}(Y)
\end{equation}
where the sum is over sets $K$ with the same cardinality as $I$ and
$J$.
\end{proof}

\begin{lemma} \label{lem:2}
Suppose $X \in \G$.  Then the rows of $X$, considered as vectors in
$\R^\infty$, are linearly independent.
\end{lemma}

\begin{proof}
Assume the statement is false and $\sum_{i \in I} p_i {\mathbf r}_i
= 0$, where $I$ is a finite set of rows, $p_i \in \R$ are real
coefficients, and ${\mathbf r}_i$ denotes the $i$-th row of $X$.
Then the rows $\r_j$ of the folding $\X$ satisfy $\sum_{i \in I} p_i
t^{i'}\r_{\bar i} = 0$, where $i'$ is defined by $i - i'\, n \in
\{1,2,\ldots,n\}$.  But this implies that the rows of $\X$ are
linearly dependent over $\R((t))$, contradicting the assumption that
$\det(\X)$ is non-vanishing.
\end{proof}

A {\it {solid minor}} of a matrix is a minor consisting of
consecutive rows and columns.  A {\it row-solid minor} (resp.~{\it
column-solid minor}) is a minor consisting of consecutive rows
(resp.~consecutive columns).

\begin{lemma} \label{lem:solid}
Suppose $X \in \G$.  Then $X$ is TNN if either all row-solid minors
of $X$, or all column-solid minors of $X$, are nonnegative.
\end{lemma}
\begin{proof}
Let $M$ be a rectangular matrix with at least as many columns as
rows.  By a theorem of Cryer \cite{Cr}, such a matrix $M$ of maximal
rank is totally nonnegative if all its row-solid minors are totally
nonnegative, cf. \cite[Theorem 2.1]{A}. By Lemma \ref{lem:2} we know
that every minor of $X$ is contained in a finite matrix of maximal
rank formed by several consecutive rows of $X$, and we may assume
that this finite matrix has more columns than rows.  Thus to
conclude nonnegativity of this minor it suffices to know
nonnegativity of the row-solid minors of $X$.  The same argument
proves the statement for column-solid minors.
\end{proof}

Throughout this paper, we will use the following naive topology on
$\G$.  Let $X^{(1)}, X^{(2)}, \ldots$ be a sequence of infinite
periodic matrices in $\G$.  Then $\lim_{k \to \infty} X^{(k)} = X$
if and only if $\lim_{k \to \infty} x^{(k)}_{ij} = x_{ij}$ for every
$i, j$.  We will show later in Proposition \ref{P:strongconverge}
that this seemingly weak notion of convergence implies much stronger
notions for convergence in the case of TNN matrices.

\begin{lem}\label{lem:limit}
Suppose $X$ is the limit of a sequence $X^{(1)}, X^{(2)}, \ldots$ of
TNN matrices.  Then $X$ is TNN.
\end{lem}
\begin{proof}
We must prove that every finite minor $\Delta_{I,J}(X)$ of $X$ is
nonnegative. But each such minor involves only finitely many
entries. Thus $\Delta_{I,J}(X) = \lim_{i \to \infty}
\Delta_{I,J}(X^{(i)}) \geq 0$.
\end{proof}

For $X, Y \in \G$, we write $X \leq Y$, if the same inequality holds
for every entry.  We note the following statement, which is used
repeatedly.

\begin{lem} \label{lem:tripleomit}
Suppose $X$, $Y$ and $Z$ are nonnegative, upper-triangular matrices
with 1's on the diagonal.  Then $XYZ \geq XZ$.
\end{lem}

\subsection{Semigroup generators for $\g_{\geq 0}$}
Let $T \subset GL_n(\R) \subset \G$ denote the subgroup of diagonal
matrices with real entries.  Let $T_{> 0}$ denote those diagonal
matrices with positive real entries. Let $S = (s_{ij})_{i,j = -\infty}^\infty \in \G$ denote the {\it
shift matrix}, defined by
$$
s_{ij} = \begin{cases} 1 & \mbox{if $j = i +1$} \\ 0 &
\mbox{otherwise.}\end{cases}
$$
The following is the loop
group analogue of the Loewner-Whitney theorem (Theorem \ref{T:LW}).

\begin{thm} \label{T:polygen}
The semigroup $\g_{\geq 0}$ is generated by shift matrices, the positive torus $T_{>
0}$ and Chevalley generators with positive parameters
$$
\{e_1(a), e_2(a), \ldots, e_{n}(a) \mid a > 0\} \ \ \cup \ \
\{f_1(a), f_2(a), \ldots, f_{n}(a) \mid a > 0\}.$$
\end{thm}

\begin{proof}
First, using a (possibly negative) power of the shift matrix we can reduce to the case when the determinant of an element of $\g_{\geq 0}$ is a non-zero real number. Next, we recall (see \cite{A}) that if
$$
M = \left( \begin{array}{cc} A & B \\ C & D \end{array}\right)
$$
is a block decomposition of a finite square matrix $M$ such that $D$
is invertible, then the {\it Schur complement} $S(M,D)$ of the block
$D$ is the matrix $A - BD^{-1} C$ which has dimensions equal to that
of $A$.

It is clear that all the generators stated in the Theorem do lie in
$\g_{\geq 0}$.  Now let $X \in \g_{\geq 0}$. Call a non-zero entry
$x_{i,j}$ of $X$ a {\it {NE corner} (northeast corner)} if
$x_{i,j+k}=x_{i-k,j}=0$ for $k \geq 1$. If $x_{i,j}$ is a NE corner
then it follows from the TNN condition for size two minors that all
entries strictly NE of $x_{i,j}$ all vanish.

A NE corner $x_{i,j}$ is {\it special} if $x_{i+1, j+1}$ is not a NE
corner.  We claim that either $x_{i,j} = 0$ for all $j>i$, or there
exists a special NE corner.  Indeed, if it was not so, that is if
all NE corners lie along a diagonal $i-j = c>0$ for some fixed $c$,
then entries on this diagonal would contribute to $\det(\X(t))$ a
monomial with a positive power of $t$ not achieved by any other term
in $\det(\X(t))$, leading to a contradiction.

Let $x_{i,j}$ be a special NE corner, which we may pick to be on a
diagonal as NE as possible.  We claim that $x_{i+1, j}
>0$. Indeed, if $x_{i+1,j}=0$ then by nonnegativity of all $2 \times
2$ minors in rows $i,i+1$ we conclude that all entries in row $i+1$
of $X$ are zero, contradicting the assumption that $X \in \g$.

Now, let $X' = e_{i}(-\frac{x_{i,j}}{x_{i+1,j}}) X$. We claim that
$X'$ is again TNN (and it is clear that $X' \in \g$).  By Lemma
\ref{lem:solid} it suffices to check nonnegativity of row-solid
minors, and in fact one only needs to check the row-solid minors
containing row $i$ of $X$ but not the row $i+1$.  Assume we have a
row-solid minor with rows $I = [i',i]$ and column set $J'$.  We may
assume that $\max(J') \leq j$, for otherwise this minor will be 0 in
both $X$ and $X'$. Now pick a set of columns $J = [j', j]$
containing $J'$.
Let $Y$ be the rectangular submatrix of $X$ with row set $[i',i+1]$
and columns set $[j',j]$. Complete it to a square matrix $Z$ by
adding zero rows or columns on the top or on the left. By
construction $Z$ is TNN and contains the row-solid minor we are
interested in. Suppose that $Z$ is a $m \times m$ matrix. Let $Z'$
be obtained from $Z$ by subtracting $\frac{x_{i,j}}{x_{i+1,j}}$
times the last row (indexed by $i+1$) from the second last row
(indexed by $i$). Then the top left $(m-1) \times (m-1)$ submatrix
of $Z'$ is by definition equal to the Schur complement of
$x_{i+1,j}$ in $Z$. It follows from \cite[Theorem 3.3]{A} that $Z'$
is also TNN, and thus the minor of $X'$ we are interested in has
nonnegative determinant.

Note that the part of the support of $X'$ above the main diagonal is
strictly contained in that of $X$.  On the other hand, the support
below the main diagonal has not increased, as can be seen by looking
again of positivity of $2 \times 2$ minors in rows $i,i+1$.  Since
after quotienting out by the periodicity the set $\supp(X)$ is
finite, this process, when repeated, must terminate.  That is, at
some point we have $x_{i,j}=0$ for all $j>i$.  A similar argument
with SW corners, and multiplication by $f_j$-s reduces $X$ to a TNN
matrix with entries only along the main diagonal.  What remains is
an element of $T_{>0}$, proving the theorem.
\end{proof}
\begin{example}
The matrix in Example \ref{ex:2} factors as $f_3(2) f_1(1) e_2(1)
e_1(1) e_3(1)$.
\end{example}

\section{Cylindric networks and total positivity} \label{sec:networks}
\subsection{Cylindric networks} \label{sec:cross}
Let $\C$ be a cylinder (that is, $S^1 \times [0,1]$) and consider an
oriented weighted network $N = (G,w,\h)$ on it defined as follows.
$G$ is a finite acyclic oriented graph embedded into $\C$, having
$n$ sources $\{v_i\}_{i=1}^n$ on one of the two boundary components
of $\C$, and having $n$ sinks $\{w_i\}_{i=1}^n$ on the other
boundary component.  Sources and sinks are numbered in
counterclockwise order (we visualize the cylinder drawn standing
with sources on the bottom and sinks on the top;
``counterclockwise'' is when viewed from above). We may, as usual,
think of the sources and sinks as labeled by $\{v_i, w_i \mid i \in
\Z/n\Z\}$ and write $v_i$ when we mean $v_{\bar i}$. The chord $\h$
is a single edge connecting the two boundary components, starting on
the arc $v_nv_1$ and ending on the arc $w_nw_1$.  We assume $\h$ is
chosen so that no vertex of $G$ lies on it.

The weight function $w: E(G) \longrightarrow \mathbb R_+$ assigns to
every edge $e$ of $G$ a real nonnegative weight $w(e)$.  The weight
$w(p)$ of a path $p$ is the product $\prod_{e \in p} w(e)$ of
weights of all edges along the path. For a collection $P = \{p\}$ of
paths we let $w(P) = \prod_{p \in P} w(p)$. For a path $p$ let the
{\it {rotor}} of $p$, denoted $\curl(p)$, be the number of times $p$
crosses $\h$ in the counterclockwise direction minus the number of
times $p$ crosses $\h$ in the clockwise direction.  If $x, y$ are
two vertices on a path $p$, we let $p_{[x,y]}$ denote the part of
the path $p$ between the points $x$ and $y$, and let $*$ denote
either the beginning or the end of a path.  For example, $p_{[x,*]}$
denotes the part of $p$ from $x$ to the end of $p$.

For an integer $i$, let us define $\alpha(i) = (i - \bar i)/n$,
where $\bar i$ is to be taken in $\{1,2,\ldots,n\}$.  For two
integers $i$ and $j$, an {\it $(i,j)$-path} is a path in $G$ which
\begin{enumerate}
 \item starts at the source $v_{\bar i}$;
 \item ends at the sink $w_{\bar j}$;
 \item has rotor equal to $\alpha(j)-\alpha(i)$.
\end{enumerate}
Define an infinite matrix $X(N) = (x_{i,j})_{i,j = -
\infty}^{\infty}$ by setting $x_{i,j}$ to be the sum of weights over
all $(i,j)$-paths in $G$.  Note that by definition $X(N)$ is
periodic: $x_{i,j} = x_{i+n,j+n}$.

Let $p$ be an $(i,j)$-path and let $q$ be an $(i',j')$-path. Assume
$c$ is a point of crossing of $p$ and $q$. Let $\tilde p$ and
$\tilde q$ be the two paths obtained by swapping $p$ and $q$ at $c$:
that is following one of them until point $c$ and the other
afterwards.  Although $\tilde p$ starts at $v_{\bar i}$ and ends at
$w_{\bar j'}$, it is not necessarily an $(i,j')$-path, since
$\curl(\tilde p)$ may not be equal to $\alpha(j') - \alpha(i)$.

\begin{figure}[h!]
    \begin{center}
    \input{loo1.pstex_t}
    \end{center}
    \caption{}
    \label{fig:loo1}
\end{figure}

\begin{example}
Let $m=n=4$ and consider two paths shown in Figure \ref{fig:loo1},
one an $(1,6)$-path and one an $(2,1)$-path. Then if we swap the two
paths at the marked point of crossing, we do not get a $(1,1)$-path
and a $(2,6)$-path.  Instead we get a $(1,5)$-path and a
$(2,2)$-path.
\end{example}

\begin{lem}\label{L:proper}
Let $c$ be a point of intersection of $p$ and $q$.  Then the path
$\tilde p$ is a $(i,j')$-path if and only if $\tilde q$ is a
$(i',j)$-path.  This happens when $\curl(p_{[c,*]}) -
\curl(q_{[c,*]}) = \alpha(j) - \alpha(j')$.
\end{lem}

In the case of Lemma \ref{L:proper}, we say that $c$ is a {\it
{proper crossing}} of $p$ and $q$.  Two paths that do not have a
proper crossing we call an {\it {uncrossed}} pair of paths.  Thus,
the crossing marked in Figure \ref{fig:loo1} is not proper. This
pair of paths is however not uncrossed since the other crossing, not
marked on the figure, happens to be proper.

\subsection{Cylindric Lindstr\"{o}m Lemma}

Let $I = i_1 < \ldots < i_K$ and $J = j_1 < \ldots < j_K$ be two
sets of indexes of equal (finite) cardinality $K$.  Let
$\Phi(I,J)$ denote the set of all families $P = \{p_k\}_{k=1}^K$
of paths such that
\begin{enumerate}
 \item each $p_k$ is an $(i_k, j_k)$-path;
 \item every pair of paths in $P$ are uncrossed.
\end{enumerate}

The following theorem is a cylindric analogue of Lindstr\"{o}m's
Lemma \cite{Li}.

\begin{theorem} \label{thm:L}
We have $$\Delta_{I,J}(X(N)) = \sum_{P \in \Phi(I,J)} w(P).$$
\end{theorem}

First we prove the following lemma.

\begin{lemma} \label{lem:1}
If $i<i'$ and $j'<j$ then every $(i,j)$-path $p$ properly crosses
every $(i',j')$-path $q$.
\end{lemma}

\begin{proof}
We make use of the following observation: assume $p$ and $q$ are two
paths that do not cross each other but might have one or two common
endpoints. Then $\curl(p)-\curl(q)$ can only take values $-1$, $0$,
or $1$.

Indeed, cut $\C$ along $p$, viewing the result as a rectangle with a
pair of opposite vertical sides identified. Since $q$ never crosses
$p$, it follows that $q$ remains strictly inside the rectangle.
Chord $\h$ is represented inside the rectangle by at least
$\curl(p)+1$ disjoint segments. We can ignore the segments which
have a crossing with the same vertical side of a rectangle, since
their intersections with $q$ contribute $0$ to $\curl(q)$. What
remains are exactly $\curl(p)+1$ segments, all but the first and the
last of which connect the two vertical sides of the rectangle. Those
$\curl(p)-1$ segments must be crossed by any path inside the
rectangle, in particular by $q$. The first and the last segments of
$p$ however may or may not be crossed, depending on relative
position of endpoints of $p$ and $q$. This implies the needed
statement concerning $\curl(p)-\curl(q)$.

We first claim that $p$ and $q$ have at least one point of
intersection.  This follows easily from unfolding the cylinder
repeatedly.  Let $c_1, \ldots, c_k$ be all the crossings of $p$ and
$q$ arranged in order.  Now, by the argument above each of the
quantities $a_0=\curl(p_{[*,c_1]})-\curl(q_{[*,c_1]})$,
$a_1=\curl(p_{[c_1,c_2]})-\curl(q_{[c_1,c_2]}), \ldots ,
a_k=\curl(p_{[c_k,*]})-\curl(q_{[c_k,*]})$ is equal to $-1$, $0$ or
$1$. Since $$\curl(p) - \curl(q) = \sum_{m=0}^k a_m =
\alpha(j)-\alpha(i)-\alpha(j')+\alpha(i') \geq \alpha(j)-\alpha(j')
\geq 0$$ there must be an index $l \in\{0,1,\ldots,k+1\}$ such that
$\sum_{m=l}^k a_m = \alpha(j)-\alpha(j')$.  If $l \notin \{0,k+1\}$
then $c_l$ is a proper crossing by Lemma \ref{L:proper}. If $l =
k+1$ then $\alpha(j) = \alpha(j')$ and $\bar j  > \bar j'$, so as a
result $a_k \leq 0$.  Similarly, if $l = 0$ then
$\alpha(i')-\alpha(i)=0$ and $\bar i'>\bar i$, so as a result $a_0
\leq 0$.  In both cases there exists at least one other index $l'
\in\{1,\ldots,k\} $ such that $\sum_{m=l}^k a_m =
\alpha(j)-\alpha(j')$. It is easy to see that the resulting $c_{l'}$
is a proper crossing.
\end{proof}

Now we are ready to prove Theorem \ref{thm:L}.

\begin{proof}
Let $P$ be a collection of $K$ paths each of which is an
$(i_k,j_l)$-path for some $k,l$ so that each element of $I$ and $J$
is used once.  Pick the first proper crossing $c$ of two paths $p, q
\in P$ (if it exists), where we choose an order on vertices of $G$
according to some height function.  We assume that the height
function is chosen so that along any path the vertices are
encountered in order of increasing height.  We can of course assume
without loss of generality that no two vertices of $G$ have the same
height. Now swap $p$ and $q$ after $c$, obtaining two new paths
$\tilde p$ and $\tilde q$. Let $\tilde P$ be the collection obtained
from $P$ by replacing $p, q$ with $\tilde p, \tilde q$.  We claim
that in $\tilde P$, $c$ is again the first proper crossing of any
pair of paths.

Assume $p$ is an $(i_k,j_l)$-path and $q$ is an
$(i_{k'},j_{l'})$-path. First, $c$ is clearly a proper crossing of
$\tilde p$ and $\tilde q$. We need to argue that it is still the
first proper crossing.  Suppose it is not.  Since $p$ and $q$ are
the only two paths in $\tilde P$ that changed, any possible new
proper crossing $\tilde c$ preceding $c$ must belong either to $p$
or to $q$ or to both.

If $\tilde c$ is a proper crossing of $\tilde p$ and $\tilde q$ then
from $\curl(p_{[c,*]}) - \curl(q_{[c,*]}) = \alpha(j_l) -
\alpha(j_{l'})$ and  $\curl(\tilde q_{[\tilde c,*]}) - \curl(\tilde
p_{[\tilde c,*]}) = \alpha(j_l) - \alpha(j_{l'})$ we obtain
$\curl(p_{[\tilde c, c]}) = \curl(q_{[\tilde c, c]})$, from which it
follows that $\tilde c$ should have been a proper crossing of $p$
and $q$ -- this contradicts the original choice of $c$.

Similarly, suppose $\tilde c$ is a proper crossing of say $\tilde q$
and some $r$, which is an $(i_{k''},j_{l''})$-path. Then
$\curl(q_{[c,*]})-\curl(p_{[c,*]})=\alpha(j_l) - \alpha(j_{l'})$ and
$\curl(\tilde q_{[\tilde c,*]})-\curl(r_{[\tilde
c,*]})=\alpha(j_{l'}) - \alpha(j_{l''})$ imply $\curl(q_{[\tilde
c,*]})-\curl(r_{[\tilde c,*]})=\alpha(j_l) - \alpha(j_{l''})$ and
$\tilde c$ should have been a proper crossing of $q$ and $r$.

Thus we have obtained a weight preserving involution on collections
$P$ of paths which have proper crossings.  We observe looking at the
corresponding terms of $\Delta_{I,J}(X(N))$ that this involution is
sign-reversing.  Thus, the corresponding contributions to the
determinant cancel.  To get the summation over $\Phi(I,J)$ it
remains to check that a collection of paths is pairwise uncrossed
only if each path in it is an $(i_k, j_k)$-path for some $k$. This
follows from Lemma \ref{lem:1}.
\end{proof}

\begin{remark}
Theorem \ref{thm:L} and the other results in this section can be
generalized to the case of $n$ sources $\{v_i\}_{i=1}^n$ and $m$
sinks $\{w_j\}_{j=1}^m$ in the obvious manner.
\end{remark}

\subsection{$\g_{\geq 0}$ and cylindric networks} \label{sec:2}


\begin{theorem} \label{thm:network}
Let $X \in \g$.  Then $X$ is equal to $X(N)$ for some cylindric
network $N$ with nonnegative weight function, if and only if $X \in
\g_{\geq 0}$.
\end{theorem}

\begin{proof}
From Theorem \ref{thm:L} it follows that every $X \in \g$ that
arises from a cylindric network is TNN. Further, concatenation of a
cylindric network $N$ and one of the special ``building block''
networks as shown in Figures \ref{fig:loo2} and \ref{fig:loo8}
corresponds to multiplication of $X(N)$ by a Chevalley generator and
by a shift matrix respectively. We conclude by Theorem
\ref{T:polygen} that every element of $g_{\geq 0}$ can be
represented by a cylindric network.
\begin{figure}[h!]
    \begin{center}
    \input{loo2.pstex_t}
    \end{center}
    \caption{}
    \label{fig:loo2}
\end{figure}
\end{proof}

\begin{figure}[h!]
    \begin{center}
    \input{loo8.pstex_t}
    \end{center}
    \caption{}
    \label{fig:loo8}
\end{figure}

\subsection{Determinant of the folding}
Let $N$ be a cylindric network.  We now give a combinatorial
interpretation for the coefficients of the determinant
$\det(\overline{X(N)}(t))$.  Let $\{v_i\}_{i=1}^n$ and
$\{w_i\}_{i=1}^n$ be the sources and sinks of $N$ as before. Then
$\overline x_{ij}(t)$ enumerates the weights of paths from $v_i$ to
$w_j$ with an extra factor $t^{\curl(p)}$ keeping track of how many
times the path $p$ crossed the chord $\h$ in the counterclockwise
direction. Let $\Gamma_k$ be the set of families $P = (p_1, \ldots,
p_n)$ of paths, satisfying: (a) the path $p_i$ connects $v_i$ and
$w_{\overline{i+k}}$, (b) no pair of paths intersect in the naive
sense (rather than in the sense of ``uncrossed'' of subsection
\ref{sec:cross}), and (c) and there are $k$ (net) counterclockwise
crossings of paths in $P$ with $\h$.

\begin{theorem}
Let $N$ be a cylindric network.  Then
$$\det(\overline {X(N)}(t)) =
\sum_{k \in \Z} \left((-1)^{k(n-1)} \sum_{P \in \Gamma_k} w(P)
\right) t^k.$$
\end{theorem}

\begin{proof}
We proceed using the usual argument in Lindstr\"om's lemma.  Suppose
$P = (p_1,p_2,\ldots,p_n)$ is a family of paths such that $p_i$ goes
from $v_i$ to $w_{\sigma(i)}$ for some permutation $\sigma \in S_n$,
and so that there are $k$ (net) counterclockwise crossings of paths
in $P$ with $\h$.  If $p_i$ and $p_j$ intersect at a vertex $c$,
swapping the two paths after $c$ will give another family $P'$ with
the same weight, and still $k$ (net) counterclockwise crossings with
$\h$. Applying the usual sign-reversing involution argument (see the
proof of Theorem \ref{thm:L}), we see that the coefficient of $t^k$
in $\det(\overline {X(N)}(t))$ is equal to the weight generating
functions of such families $P$ with the additional requirement that
no pair of paths intersect.  We now observe such families $P$ exist
only if $\sigma$ is a power of the long cycle, that is, belong to
$\Gamma_k$.  The sign of the corresponding permutation $\sigma$ is
$(-1)^{k(n-1)}$.
\end{proof}

\begin{example} \label{ex:3}
Consider the network  given in Figure \ref{fig:loo6}, where all
edges are oriented upwards and have weight $1$.
\begin{figure}[h!]
    \begin{center}
    \input{loo6.pstex_t}
    \end{center}
    \caption{}
    \label{fig:loo6}
\end{figure}
One can check that the associated element of $\G$ and its folding are given by
$$
\left(\begin{array}{ccccccc} \ddots & \vdots & \vdots & \vdots & \vdots &\vdots \\
\cdots& 3& 5 & 2 & 1 & 0& \cdots \\ \cdots& 1& 7& 4 & 2 & 0 &
 \cdots \\
 \cdots& 0&3 & 3& 5& 2&\cdots \\
 \cdots& 0&1&1&7& 4& \cdots \\
  \cdots& 0&0&0&3&3&  \cdots \\
 & \vdots & \vdots & \vdots & \vdots &\vdots &\ddots
\end{array} \right)
\rightsquigarrow
\left(\begin{array}{cc}
3+2t & 3t^{-1}+5+t \\
1+4t & t^{-1}+7+2t
\end{array} \right)
.
$$
The determinant of the folded matrix equals $6-t$.  The non-crossing subnetwork corresponding to the $-t$ term is shown on the right of Figure \ref{fig:loo6}.
\end{example}

\begin{corollary}\label{C:coeffs}
If $X = X(N)$ arises from a cylindric network $N$, then the odd
minors of $\overline X(t)$ have nonnegative coefficients, the even
minors have sign-alternating coefficients.
\end{corollary}

\section{Upper triangular matrices and a reduction result}
\label{sec:upper}
\subsection{Upper triangular matrices}
Let $U \subset \G$ be the subgroup of the formal loop group
consisting of infinite periodic matrices which are upper triangular,
and such that all diagonal entries are equal to 1.  We denote the
totally nonnegative matrices in $U$ by $U_{\geq 0}$, and the totally
positive matrices in $U$ by $U_{> 0}$.

We say that $X \in U_{\geq 0}$ is {\it finitely supported} if
finitely many of diagonals of $X$, given by $j - i = {\rm
constant}$, are non-zero. Otherwise we say that $X$ is not finitely
supported.

\begin{lem} \label{lem:fs}
If $X \in U_{\geq 0}$ is not finitely supported then all of its
entries above the main diagonal are non-zero.
\end{lem}
\begin{proof}
Suppose some entry $x_{i,j} = 0$.  By using the nonnegativity of the
$2 \times 2$ minors involving $x_{i,j}$ and either $x_{i,i}$ or
$x_{j,j}$ we deduce that $x_{i,k} = 0$ for $k > j$ and $x_{k,j} = 0$
for $k < i$.  Thus all the entries northeast of $x_{i,j}$ are 0.
Since the entries of $X$ are periodic, we deduce that $X$ is
finitely supported.
\end{proof}

The entries of the folding of a totally positive $X$ are thus
polynomials if $X$ is finitely supported and infinite power series
otherwise.

\subsection{Reduction to $U_{\geq 0}$}

\begin{thm}\label{T:red}
Every $X \in \G_{\geq 0}$ has a unique factorization of the form $X
= F S^k Y$ where $F$ is the product of an element in $T_{>0}$ and
some $f_i(a)$-s, $k$ is an integer, and $Y \in U_{\geq 0}$.
\end{thm}
\begin{proof}
We first prove existence.  By the definition of $\G$, the matrix $X$
has at least one SW-corner, where SW-corner is defined in obvious
analogy with the NE-corners used in the proof of Theorem
\ref{T:polygen}.  Arguing as in that proof, either (a) one can write
$X = f_{j}(a) X'$ where $X' \in \G$ and $a > 0$, or (b) the
southwestmost non-zero of diagonal of $X$ is completely filled with
non-zero entries.  If we are in Case (b), then we can use the shift
matrix $S$ to shift the southwest-most diagonal to the central
diagonal, and then multiply by a matrix in $T_{>0}$ to obtain the
desired matrix $Y \in U_{\geq 0}$.  In Case (a), we repeatedly
factor out Chevalley generators $f_j(a)$, which in particular does
not change the determinant $\det(X)$.  We must eventually encounter
Case (b), for otherwise we will have reduced the support of $X$ to
so far in the northeast that the lowest degree monomial in $\det(X)$
cannot be obtained.  This establishes existence.

We now prove uniqueness.  We first note that $S T_{> 0} S^{-1} \in
T_{>0}$ and that $S f_{i}(a) S^{-1} = f_{i-1}(a)$.  Suppose we have
$F S^k Y = F' S^{k'} Y'$. Then one has $Y'' = S^{k''} F''$ where
$Y'' \in U$, $F''$ is a product of $f_i(a)$-s with possibly negative
parameters, and $k'' \in \Z$.  But $\det(Y'') \in 1 + t\R[[t]]$ and
$\det(F'') \in \R$, so we conclude that $k'' = 0$.  But $F''$ is
lower triangular, and $Y''$ is upper triangular, so $F'' = Y'$ is
the identity matrix.  This implies that $k = k'$, $F = F'$, and $Y =
Y'$.
\end{proof}

For the rest of this section, and most of the rest of the paper, we
focus on the semigroup $U_{\geq 0}$.

\subsection{Convergence in $U_{\geq 0}$}
A {\it totally positive function} is a formal power series $a(t) = 1
+ a_1t + a_2 t^2 + \cdots$ which arises as $a(t) = \X(t)$ for $X \in
U_{\geq 0}$ with $n = 1$.  Note that with this terminology, we do
not make the usual distinction between totally nonnegative and
totally positive.  As we have mentioned, the Edrei-Thoma theorem
(Theorem \ref{T:ET}) classifies totally positive functions.

%

\begin{prop}\label{prop:mero}
Suppose $X \in U_{\geq 0}$.  Then the entries of $\overline X(t)$
are meromorphic functions holomorphic in a neighborhood of 0.
\end{prop}
\begin{proof}
Apply Theorem \ref{T:ET} to each entry of $\X(t)$.  (See also the
proof of Proposition \ref{P:strongconverge}.)
\end{proof}

The {\it {radius of convergence}} of $X$, denoted $r(X)$, is the
minimum of the radii of convergence of the entries of $\X(t)$.  The
following Proposition shows that our weak notion of convergence
automatically implies stronger convergence.

\begin{prop}\label{P:strongconverge}
Suppose $X^{(1)}, X^{(2)}, \ldots$ is a sequence of matrices in
$U_{\geq 0}$ with limit $X$.  Then there is a neighborhood $V
\subset \cc$ of 0 so that
\begin{enumerate}
\item every matrix amongst $\X^{(i)}(t)$ and $\X(t)$ is holomorphic in $V$
\item
every matrix entry of $\X^{(i)}(t)$ approaches the corresponding
entry of $\X(t)$ uniformly, considered as holomorphic functions on
V.
\end{enumerate}
\end{prop}
\begin{proof}
It is enough to prove the statement for the case $n = 1$, that is,
for totally positive functions.  If $a(t) = 1 + a_1 t + \cdots$ is a
totally positive function, then looking at $2 \times 2$ minors we
have $a_1 \geq a_2/a_1 \geq a_3/a_2 \geq \cdots$, whenever the
ratios are defined.  Thus if $a(t)$ is not a polynomial, the radius
of convergence $r(a)$ of $a(t)$ is at least $a_i/a_{i+1}$ and we
have $r = \lim_{i \to \infty} a_i/a_{i+1}$.

Now suppose that $a^{(1)}(t), a^{(2)}(t), \ldots$ converge to
$a(t)$.  Then there is a sufficiently large $N$ so that for $k > N$,
$|a^{(k)}_1 - a_1| \leq 1$.  It follows that $r(a^{(k)}(t)) >
1/(a_1+1)$ for all $k > N$ and so there exists a neighborhood $V$ of
0 with property (1).

To see that $a^{(i)}(t)$ approaches $a(t)$ uniformly in a possibly
smaller neighborhood $V$, we note that for $|t| < R$ we have
$$
\left|\sum_{i \geq k} a_i t^i\right| \leq a_k R^k \sum_{i \geq k}
a_1^{i-k} R^{i-k} \leq \frac{(a_1R)^k}{1-a_1 R}.
$$
Fix some $R \ll 1/a_1$.  It follows that for any $\ell \gg 0$, the
value of $|a(t) - a^{(\ell)}(t)|$ for $|t| < R$ can be approximated
by throwing away all but the first $k$ terms.  But for $\ell$
sufficiently large, the first $k$ terms of $a(t)$ and
$a^{(\ell)}(t)$ are arbitrarily close.  This shows that $a^{(i)}(t)$
approaches $a(t)$ uniformly in $|t| < R$.
\end{proof}

Note that neither conclusion of Proposition \ref{P:strongconverge}
holds for general meromorphic functions.

%
%
%

\subsection{The operation $^{-c}$}
We define $X^c \in U$ to be the matrix obtained by applying to $X
\in U$ the transformation $x_{i,j} \mapsto (-1)^{|i-j|}x_{i,j}$. A
special role in what follows is played by the operation $^c$-{\it
{inverse}} given by $X \mapsto (X^c)^{-1}$.  Abusing notation
slightly, we shall also write $X^{-c} := (X^c)^{-1}$.  Note that
$(X^c)^{-1} = (X^{-1})^c$.  Also note that the operation $X \mapsto
X^{-c}$ is an involution, and that $(XY)^{-c} = Y^{-c}X^{-c}$.

\begin{lemma}\label{lem:cinverse}
Suppose $X \in U_{\geq 0}$.  Then $X^{-c} \in U_{\geq 0}$.
\end{lemma}
\begin{proof}
It suffices to show that $X^{-c}_{I,I}$ is TNN for every interval $I
= [a,b]$, since every minor of $X^{-c}$ is contained in such a
submatrix.  Let $Y = X_{I,I}$ and $m = |I|$.  Then $Y \in
GL_m(\R)_{\geq 0} \subset GL_m(\R[t,t^{-1}])_{\geq 0}$.  By Theorem
\ref{T:polygen} (or Theorem \ref{T:LW}), $Y$ is a product of
Chevalley generators $\{e_i(a) \mid i = 1,2,\ldots,m-1\}$ with
positive parameters. We now observe that $e_i(a)^{-c} = e_i(a)$.
Using $(WV)^{-c} = W^{-c}V^{-c}$, we deduce that $Y^{-c}$ is also a
product of Chevalley generators with positive parameters.  But then
$X^{-c}_{I,I} = Y^{-c}$ is TNN.
\end{proof}

Suppose $i, j, k$ are integers such that $j-i-k \geq -1$ and $k \geq
0$.  Let $X_{i,j,k}$ denote the solid submatrix of $X$ obtained from
the rows $i, i+1, \ldots, j-k$ and the columns $i+k, i+k+1, \ldots,
j$.
\begin{prop}
Let $X \in U$.  Then $\det(X_{i,j,k}) = \det(X^{-c}_{i,j,j+1-i-k})$,
where if $j = i + k -1$ we define $\det(X_{i,j,k}) = 1$.
\end{prop}
\begin{proof}
If $k = 0$, then $\det(X_{i,j,k}) = 1 = \det(X^{-c}_{i,j,j+1-i-k})$.
Consider now $k = 1$.  That is, we need to show $(X^{-c})_{i,j} =
\det(X_{i,j,1})$. Expanding $\det(X_{i,j,k})$ into smaller minors
using the first row, we obtain
$$
\det(X_{i,j,1}) = \sum_{r = 0}^{j-i-1} (-1)^{r} x_{i,i+r+1}
\det(X_{i+r+1,j,1}).
$$
The claim then follows from the definition of $X^{-c}$ and induction
on $j - i$.

We now allow $k$ to be arbitrary.  We will prove the equality as a
polynomial identity. Recall that for an $n \times n$ matrix $M$,
Dodgson's condensation lemma \cite{D} says
\begin{align}\label{E:Dodgson} &\Delta(M)\Delta_{\{2,3,\ldots,n-1\},\{2,3,\ldots,n-1\}}(M) = \\
&\Delta_{\{1,\ldots,n-1\},\{1,\ldots,n-1\}}(M)\Delta_{\{2,\ldots,n\},\{2,\ldots,n\}}(M)
- \Delta_{\{2,\ldots,n\},\{1,\ldots,n-1\}}(M)
\Delta_{\{1,\ldots,n-1\},\{2,\ldots,n\}}(M). \notag
\end{align}

Applying this and proceeding by induction on $k$, we calculate
\begin{align*}&\det(X_{i,j,k+1}) \\ &= \frac{\det(X_{i,j-1,k})\cdot
\det(X_{i+1,j,k})-\det(X_{i,j,k})\cdot
\det(X_{i+1,j-1,k})}{\det(X_{i+1,j-1,k-1})} \\
& =\frac{\det(X^{-c}_{i,j-1,j-i-k})\cdot
\det(X^{-c}_{i+1,j,j-i-k})-\det(X^{-c}_{i,j,j+1-i-k})\cdot
\det(X^{-c}_{i+1,j-1,j-1-i-k})}{\det(X^{-c}_{i+1,j-1,j-i-k})}\\
&= \det(X^{-c}_{i,j,j-i-k}).
\end{align*}
Note that the equalities hold as polynomials when applied to a
matrix $X$ consisting of variables $x_{i,j}$.  Thus the divisions in
the calculation are always legitimate.
\end{proof}

\begin{lem}\label{lem:detfold}
We have
$$
\det(\overline{X^c})(t) = \begin{cases} \det(\overline X)(t) &
\mbox{if $n$ is even} \\ \det(\overline X)(-t) & \mbox{if $n$ is
odd.}
\end{cases}
$$
\end{lem}
\begin{proof}
Suppose $\overline{X}(t) = (\overline{x}_{ij}(t))$.  Then
$\overline{X^c}(t)$ has entries $(-1)^{j-i}\overline{x}_{ij}((-1)^n
t)$.
\end{proof}

\section{Whirls, curls, and ASW factorization}\label{sec:ASW}
\subsection{Whirls and curls}
Let $a_1,\ldots,a_n$ be $n$ real parameters.  We define a {\it
{whirl}} to be a matrix $M = (m_{i,j})_{i,j=-\infty}^{\infty} =
M(a_1, \ldots, a_n)$ with $m_{i,i}=1$, $m_{i,i+1}=a_{i}$ and the
rest of the entries equal to zero.  Here, the indexing of the
parameters are taken modulo $n$.  Note that the Chevalley generator
$e_i(a)$ is given by $M(0,\ldots,0,a,0,\ldots,0)$ where the $a$ is
in the $i$-th position.  If at least one of the parameters $a_i$ in
a whirl is zero, then we call the whirl {\it {degenerate}}.  A
degenerate whirl always factors into Chevalley generators.
Furthermore, if the original parameters are nonnegative then the
parameters in factorization are also nonnegative.  We define a {\it
{curl}} to be a matrix $N$ of the form $N(a_1, \ldots, a_n) :=
M(a_1, \ldots, a_n)^{-c}$.  Examples of whirls and curls were given
in Section \ref{sec:intro}.

\begin{lem}\label{lem:whirldet}
The folded determinants of whirls and curls are given by
\begin{align*}
\det(M(a_1,\ldots,a_n)) = 1 + (-1)^{n+1} (\prod_{i=1}^n a_i)\, t \ \
\ \ \det(N(a_1,\ldots,a_n)) = \frac{1}{1 - (\prod_{i=1}^n a_i)\, t}.
\end{align*}
\end{lem}

\subsection{$\epsilon$-sequence}
Let $X \in U_{\geq 0}$.  Define
$$\epsilon_i = \epsilon_i(X) = \lim_{j \longrightarrow \infty}
\frac{x_{i,j}}{x_{i+1,j}}.$$  Clearly $\epsilon_i$ depends only on $\bar i$.
Similarly, define
$$
\mu_i = \mu_i(X) = \lim_{j \longrightarrow -\infty}
\frac{x_{j,i+1}}{x_{j,i}}.
$$
\begin{example} \label{ex:4}
Let $n = 2$.  Consider the following matrix.
$$
\left(\begin{array}{cc}
\frac{1+2t}{(1-t)(1-2t)} & \frac{2}{(1-t)(1-2t)} \\
\frac{3t}{(1-t)(1-2t)} & \frac{1+t}{(1-t)(1-2t)}
\end{array} \right)
\rightsquigarrow
\left(\begin{array}{ccccccc} \ddots & \vdots & \vdots & \vdots & \vdots &\vdots \\
\cdots& 1& 2 & 5 & 6 & 13& \cdots \\ \cdots& 0& 1& 3 & 4 & 10 &
 \cdots \\
 \cdots& 0&0 & 1& 2& 5&\cdots \\
 \cdots& 0&0&0&1& 3& \cdots \\
  \cdots& 0&0&0&0&1&  \cdots \\
 & \vdots & \vdots & \vdots & \vdots &\vdots &\ddots
\end{array} \right)
.
$$
This matrix is in fact the product $N(1,1) N(1,2)$ of two curls, and
thus is totally nonnegative. Then $\epsilon_1 = \lim_{i \to \infty}
\frac{2^{i+2}-3}{3(2^i-1)} = \frac{4}{3}$. Similarly one computes
$\epsilon_2 = \frac{3}{2}$.
\end{example}

\begin{lemma}\label{lem:limits}
Suppose $X \in U_{\geq 0}$ and not finitely supported.  Then the limits
$\epsilon_i$ and $\mu_i$ exist.  Furthermore, $1/(\prod_{i=1}^n
\epsilon_i) = 1/(\prod_{i=1}^n \mu_i)$ is the radius of convergence
of every entry of the folding $\overline X(t)$.
\end{lemma}

\begin{proof}
The inequality $\frac{x_{i,j}}{x_{i+1,j}} \geq
\frac{x_{i,j+1}}{x_{i+1,j+1}}$ follows from the nonnegativity of the
$2 \times 2$ minor $x_{i,j}x_{i+1,j+1} - x_{i+1,j}x_{i,j+1}$ of $X$.
A non-increasing sequence of nonnegative real numbers has a limit,
giving the first statement of the Lemma.  The second statement
follows from the observation that
$$\frac{x_{i,j+n}}{x_{i,j}}=\frac{x_{i-n,j}}{x_{i,j}} =
\prod_{k=0}^{n-1}\frac{x_{i+k-n,j}}{x_{i+k+1-n,j}}.$$
\end{proof}

Although we often omit it from notation, the $\epsilon_i$-s are
depend on $X$.  We call $(\epsilon_1,\ldots,\epsilon_n)$ the
$\epsilon$-sequence of $X$.  Aissen, Schoenberg, and Whitney
\cite{ASW} used a factorization procedure as a first step towards
the Edrei-Thoma theorem.  We now describe a generalization of it to
$n >1$.  We call this generalization {\it {ASW factorization}}.

\begin{lemma}\label{lem:ASW}
Suppose $X \in U_{\geq 0}$ is not finitely supported.  Let $X' =
M(-\epsilon_1, \ldots, -\epsilon_n) X$.  Then $X' \in U_{\geq 0}$.
\end{lemma}

\begin{proof}
Let $J = j_1 < j_2 < \cdots < j_k$ be a set of column indices.  We
have
$$
\lim_{l \longrightarrow \infty}
\frac{
\det \left(
\begin{matrix}
x_{i,j_1} &  x_{i,j_2}  & \ldots & x_{i,j_k} & x_{i,l}\\
x_{i+1,j_1} &  x_{i+1,j_2}  & \ldots & x_{i+1,j_k} & x_{i+1,l}\\
\vdots & \vdots & \ddots & \vdots & \vdots\\
x_{i+k,j_1} &  x_{i+k,j_2}  & \ldots & x_{i+k,j_k} & x_{i+k,l}\\
x_{i+k+1,j_1} &  x_{i+k+1,j_2}  & \ldots & x_{i+k+1,j_k} &
x_{i+k+1,l}
\end{matrix}
\right)}
{x_{i+k+1,l}}=
$$

$$
\lim_{l \longrightarrow \infty}
\det \left(
\begin{matrix}
x_{i,j_1} &  x_{i,j_2}  & \ldots & x_{i,j_k} & x_{i,l}/x_{i+k+1,l}\\
x_{i+1,j_1} &  x_{i+1,j_2}  & \ldots & x_{i+1,j_k} & x_{i+1,l}/x_{i+k+1,l}\\
\vdots & \vdots & \ddots & \vdots & \vdots\\
x_{i+k,j_1} &  x_{i+k,j_2}  & \ldots & x_{i+k,j_k} & x_{i+k,l}/x_{i+k+1,l}\\
x_{i+k+1,j_1} &  x_{i+k+1,j_2}  & \ldots & x_{i+k+1,j_k} & 1
\end{matrix}
\right)=
$$

$$
\det \left(
\begin{matrix}
x_{i,j_1} &  x_{i,j_2}  & \ldots & x_{i,j_2} & \epsilon_{i}\dotsc\epsilon_{i+k}\\
x_{i+1,j_1} &  x_{i+1,j_2}  & \ldots & x_{i+1,j_2} & \epsilon_{i+1}\dotsc\epsilon_{i+k}\\
\vdots & \vdots & \ddots & \vdots & \vdots\\
x_{i+k,j_1} &  x_{i+k,j_2}  & \ldots & x_{i+k,j_2} & \epsilon_{i+k}\\
x_{i+k+1,j_1} &  x_{i+k+1,j_2}  & \ldots & x_{i+k+1,j_2} & 1
\end{matrix}
\right)=
$$

$$
\det \left(
\begin{matrix}
x_{i,j_1}-\epsilon_{i}x_{i+1,j_1} &  x_{i,j_2}-\epsilon_{i}x_{i+1,j_2}  & \ldots & x_{i,j_k}-\epsilon_{i}x_{i+1,j_k} & 0\\
x_{i+1,j_1}-\epsilon_{i+1}x_{i+2,j_1} &  x_{i+1,j_2}-\epsilon_{i+1}x_{i+2,j_2}  & \ldots & x_{i+1,j_k}-\epsilon_{i+1}x_{i+2,j_k} & 0\\
\vdots & \vdots & \ddots & \vdots & \vdots\\
x_{i+k,j_1}-\epsilon_{i+k}x_{i+k+1,j_1} &  x_{i+k,j_2}-\epsilon_{i+k}x_{i+k+1,j_2}  & \ldots & x_{i+k,j_k}-\epsilon_{i+k}x_{i+k+1,j_k} & 0\\
x_{i+k+1,j_1} &  x_{i+k+1,j_2}  & \ldots & x_{i+k+1,j_k} & 1
\end{matrix}
\right)=
$$

$$
\det \left(
\begin{matrix}
x_{i,j_1}-\epsilon_{i}x_{i+1,j_1} &  x_{i,j_2}-\epsilon_{i}x_{i+1,j_2}  & \ldots & x_{i,j_k}-\epsilon_{i}x_{i+1,j_k} \\
x_{i+1,j_1}-\epsilon_{i+1}x_{i+2,j_1} &  x_{i+1,j_2}-\epsilon_{i+1}x_{i+2,j_2}  & \ldots & x_{i+1,j_k}-\epsilon_{i+1}x_{i+2,j_k} \\
\vdots & \vdots & \ddots & \vdots & \\
x_{i+k,j_1}-\epsilon_{i+k}x_{i+k+1,j_1} &
x_{i+k,j_2}-\epsilon_{i+k}x_{i+k+1,j_2}  & \ldots &
x_{i+k,j_k}-\epsilon_{i+k}x_{i+k+1,j_k}
\end{matrix}
\right).
$$

This is a minor of $X'$, and every row-solid minor of $X'$ can be
presented as a limit in this way.  Since a limit of a nonnegative
quantity is nonnegative, we conclude that all row-solid minors of
$X'$ are nonnegative. By Lemma \ref{lem:solid}, we conclude that
$X'$ is totally nonnegative.
\end{proof}

We can rewrite the definition of $X'$ as $X = N(\epsilon_1, \ldots,
\epsilon_n) X'$.  This gives a factorization of $X$ into a product
of two TNN matrices.  Note that the radius of convergence of $X'$ is
at least as large as that of $X$.  Thus, if we repeat the ASW
factorization to obtain $X = N(\epsilon_1, \ldots, \epsilon_n)
N(\epsilon'_1, \ldots, \epsilon'_n) X''$ then we must have
$\prod_{i=1}^n \epsilon_i \geq \prod_{i=1}^n \epsilon'_i$. We also
note that the factorization in Lemma \ref{lem:ASW} involves the
``biggest'' whirl.

\begin{lemma}\label{lem:ASWconv}
Suppose $X \in U_{\geq 0}$ is not finitely supported. Suppose that
$$X' = M(-a_1, \ldots, -a_n) X$$ is TNN.  Then $a_i \leq \epsilon_i$
for each $i$.  Furthermore, if $a_i < \epsilon_i$ for some $i$ then
$r(X') = r(X)$.
\end{lemma}
\begin{proof}
We obtain $X' = (x'_{i,j})$ from $X = (x_{i,j})$ by subtracting
$a_i$ times the $(i+1)$-th row from the $i$-th row.  But the ratio
$x_{i,j}/x_{i+1,j}$ approaches $\epsilon_i$, so $x'_{i,j} \geq 0$
implies that $a_i \leq \epsilon_i$.

For the last statement, suppose that $a_i < \epsilon_i$.  Since
$r(M(-a_1,\ldots,-a_n)) = \infty$, we have $r(X') \geq r(X)$.  But
using Lemma \ref{lem:limits}, we have $r(N(a_1,\ldots,a_n)) =
\prod_i\frac{1}{a_i} > \prod_i \frac{1}{\epsilon_i} = r(X)$ so that
from $X = N(a_1,\ldots,a_n)X'$, we have $r(X) \geq r(X')$.  Thus
$r(X') = r(X)$.
\end{proof}

\begin{example} \label{ex:5}
In Example \ref{ex:4}, it was computed that the curl $N(\frac{4}{3},
\frac{3}{2})$ can be factored out on the left. One can check that
the remaining totally nonnegative matrix is the curl $N(\frac{2}{3},
\frac{3}{2})$. Thus the ASW factorization of the matrix in this
example is $N(\frac{4}{3}, \frac{3}{2}) N(\frac{2}{3},
\frac{3}{2})$.
\end{example}

\subsection{Finitely supported TNN matrices}
\begin{theorem}\label{thm:finitefactor}
The semigroup $U^{\rm fin}_{\geq 0}$ of finitely supported matrices
in $U_{\geq 0}$ is generated by whirls and Chevalley generators with
nonnegative parameters.
\end{theorem}

\begin{proof}
It is clear that the semigroup generated by whirls and Chevalley
generators with nonnegative parameters lies inside $U^{\rm
fin}_{\geq 0}$.  Now let $X \in U^{\rm fin}_{\geq 0}$.  First
suppose that $X^{-c}$ is finitely supported.  In this case, the
entries of $\overline{X^{-c}}(t)$ are polynomials, and in
particular, entire. But then both $1/\det(\overline{X^c}(t)) =
\det(\overline{X^{-c}}(t))$ and $\det(\overline{X^c}(t))$ are
polynomials, so we conclude that $\det(\overline{X^c}(t))$ and by
Lemma \ref{lem:detfold} $\det(X)$ is a constant.  By Theorem
\ref{T:polygen}, we deduce that $X$ fan be factored into a finite
number of nonnegative Chevalley generators.

Now suppose that $X^{-c}$ is not finitely supported.  Apply Lemma
\ref{lem:ASW} and Lemma \ref{lem:cinverse} to obtain $X^{-c} =
N(a_1, \ldots, a_n) Y$, where the parameters $a_i =
\epsilon_i(X^{-c})$ are nonnegative and $Y$ is totally nonnegative.
If at least one of parameters $a_i$ is zero, by Lemma
\ref{lem:limits} the entries of $\overline{X^{-c}}$ are entire, and
the determinant is entire.  We may then proceed as in the case that
$X^{-c}$ is finitely supported.

Thus we may assume that all $a_i$ are strictly positive. Then $X =
Y^{-c} M(a_1, \ldots, a_n)$, where both $X$ and $Y^{-c}$ are
finitely supported TNN matrices.  One observes that the number of
non-zero diagonals of $Y^{-c}$ must be strictly smaller than that of
$X$. Now repeat the application of Lemma \ref{lem:ASW} to $Y^{-c}$.
Since the number of non-zero diagonals of $X$ is finite, in a finite
number of steps we must obtain the situation in one of the two
previous paragraphs.  Thus we obtain a factorization of $X$ into a
finite number of whirls and Chevalley generators with nonnegative
parameters.
\end{proof}

Since whirls are representable by cylindric networks, as shown on
the left in Figure \ref{fig:loo3}, we immediately get the following
corollary.
\begin{figure}[h!]
    \begin{center}
    \input{loo3.pstex_t}
    \end{center}
    \caption{}
    \label{fig:loo3}
\end{figure}
\begin{corollary}
Every $X \in U^{\rm fin}_{\geq 0}$ is representable by a cylindric
network.
\end{corollary}

\subsection{Totally positive matrices}
For $I = \{i_1 < i_2 < \cdots < i_k\}$ and $J =\{j_1 < j_2 < \cdots
< j_k\}$ we define $I \leq J$ if $i_t \leq j_t$ for each $t \in
[1,k]$.

\begin{theorem} \label{thm:tp}
Let $X \in U_{\geq 0}$.  Then $X \notin U_{>0}$ if and only if $X$
is a finite product of whirls and curls (including Chevalley
generators).  In other words, the semigroup generated by whirls and
curls is exactly the set $U_{\geq 0} - U_{> 0}$.
\end{theorem}

We start by proving the following lemma.

\begin{lemma}\label{lem:sol}
Suppose $X$ has a vanishing minor $\Delta_{I,J}(X)=0$ for $I \leq
J$. Assume that $(I,J)$ is chosen so that $|I| = |J| = k$ is
minimal.  Then $X$ has a solid vanishing minor $\Delta_{I',J'}(X) =
0$ of size $k$ with $I' \leq J'$. Furthermore, all minors
$\Delta_{I'',J''}(X)$ for $I'' \leq I'$ and $J' \leq J''$ vanish.
\end{lemma}

\begin{proof}
For $k=1$ the statement is already proved in Lemma \ref{lem:fs}, so
assume $k>1$. If $I = i_1 < \ldots < i_k$ and $J = j_1 < \ldots <
j_k$ then $i_k < j_k$ since otherwise there is a smaller singular
minor. Look at the submatrx $X_{I \cup \{j_{k}\}, J \cup
\{j_k+1\}}$. Writing down Dodgson's condensation \eqref{E:Dodgson}
for this matrix we get
$$- \Delta_{I \cup \{j_k\} -\{i_1\}, J}(X) \Delta_{I, J\cup \{j_k+1\} - \{j_1\}}(X) =
\Delta_{I \cup \{j_{k}\}, J \cup \{j_k+1\}}(X) \Delta_{I-\{i_1\},
J-\{j_1\}}(X).$$ This implies that the left-hand side must be zero,
since it is non-positive and the right-hand side is nonnegative. If
$\Delta_{I \cup j_k - \{i_1\}, J}(X)=0$ then the size $k-1$ minor
$\Delta_{I - \{i_1\}, J- \{j_k\}}(X)$ vanishes.  If $I' = I -
\{i_1\}$ and $J' = J - \{j_k\}$ satisfies $I' \leq J'$ then this
contradicts the minimality of $k$.  Otherwise we would have $i_{t+1}
> j_t$ for some $t \in [1,k-1]$, implying that the submatrix
$X_{I,J}$ is block upper triangular.  Again this would imply a
smaller vanishing minor, contradicting the minimality of $k$.

Thus $\Delta_{I, J\cup \{j_k+1\}-\{j_1\}}(X)=0$.  Repeating this $k$
times, the column indexing set becomes solid, and similarly, we may
move the rows up to obtain a solid row indexing set.  The second
claim is proved in a similar manner.
\end{proof}

\begin{cor}\label{C:TP}
Suppose $X \in U_{\geq 0}$.  Then $X \in U_{> 0}$ if and only if all
minors $\Delta_{I,J}(X) > 0$ for $I \leq J$.
\end{cor}

\begin{lemma} \label{lem:fsin}
If $X \in U_{\geq 0}$ is not finitely supported and has a vanishing
solid minor then all $\epsilon_i$-s are positive.
\end{lemma}

\begin{proof}
Let $k>1$ be the size of smallest singular minor.  It was shown in
Lemma \ref{lem:sol} that all, not necessarily solid, minors of size
$k$ far enough from the diagonal are singular.  Consider the $k
\times \infty$ submatrix $Y = X_{I,J}$ where $I =
\{i,i+n,\ldots,i+(k-1)n\}$ and $J = \{j,j+n,\ldots\}$, where $i$ and
$j$ are chosen so that all $k \times k$ minors of $Y$ vanish. Thus
$Y$ has rank less than $k$.  Since the $(k-1) \times (k-1)$ minors
of $Y$ do not vanish, there is a unique (up to scalar factor) linear
relation between the rows of $Y$, say $\sum_{r=1}^k c_r {\bf y}_r =
0$, where $y_r$ is the $r$-th row of $Y$, and all the $c_r$ are
non-zero.

We deduce that for large enough $t$ we have $ \sum_{r=1}^{k} c_r
x_{i,j+(t-r+1)n}$.  Then the limit $\delta = \lim_{t \to \infty}
\frac{x_{i,j+tn}}{x_{i,j+(t-1)n}}$, which we know exists by Lemma
\ref{lem:limits}, satisfies the polynomial equation $\sum_{r=1}^k
c_r \delta^{k-r} = 0$.  Since the $c_r$ (in particular $c_k$) are
all non-zero, $\delta \neq 0$.  But $\delta$ is exactly the product
of all $\epsilon_i$-s (for $i = 1,2,\ldots,n$).
%
%
\end{proof}

Now we are ready to prove the theorem.

\begin{proof}[Proof of Theorem \ref{thm:tp}]
Whirls, curls, and Chevalley generators all have the property that
minors sufficiently far from the diagonal vanish.  Thus any finite
product of such matrices will have the same property.  This shows
that the semigroup generated by whirls and curls consists of totally
nonnegative but not totally positive matrices.

Now suppose $X \in U_{\geq 0}$ is not totally positive.  By
Corollary \ref{C:TP}, $X$ has a vanishing minor $\Delta_{I,J}(X) =
0$ for $I \leq J$, which by Lemma \ref{lem:sol} we may assume to be
solid.  We first suppose that $(I,J)$ is chosen so that $I \leq J -
1$ and $\Delta_{I,J - 1}(X) > 0$ (here $J - 1$ denotes $\{j_1 - 1,
\ldots, j_k- 1\}$ where $J = \{j_1,\ldots,j_k\}$).  This is possible
because if $I$ is not $\leq J -1$, and both $I, J$ are solid then $I
= J$ and $\Delta_{I,J}(X)$ cannot vanish.

If $X^{-c}$ is finitely supported, the statement follows from
Theorem \ref{thm:finitefactor}. If it is not finitely supported, we
claim that ASW factorization (Lemma \ref{lem:ASW}) factors a
non-degenerate curl from $X^{-c}$. For that first note that if $I =
(i+1, \ldots, i+k)$ and $J = (j+1, \ldots, j+k)$ then as was shown
in the proof of Lemma \ref{lem:cinverse} $\Delta_{I,J}(X) =
\Delta_{I',J'}(X^{-c})$ where $I' = (i+1, \ldots, j)$ and $J' =
(i+1+k, \ldots, j+k)$. Thus $X^{-c}$ also has a singular solid minor
with $I' \leq J'$.  By Lemmata \ref{lem:fsin} and \ref{lem:ASW}, a
non-degenerate curl $N$ can be factored out from $X^{-c}$.  We may
thus write $X = X' M$ for a whirl $M = N^{-c}$ and totally
nonnegative $X'$. We claim that in $X'$ the minor $X'_{I,J-1}$ is
singular. Indeed, in $M$ the minor $M_{J-1,J}$ is non-singular. Then
if $\Delta_{I,J-1}(X')>0$ then by the Cauchy-Binet formula
\eqref{E:CB} we would have a positive term contributing to
$\Delta_{I,J}(X)$, and since all other terms are nonnegative we
obtain a contradiction.

Repeating this argument, the vanishing minor of $X$ is moved closer
and closer to the diagonal, so the process must eventually stop, at
which point we will have obtained the desired factorization of $X$.
\end{proof}
%
%

Note that curls can be represented by (non-acyclic) cylindric networks
as shown on the right in Figure \ref{fig:loo3}.  The definitions and
results of Section \ref{sec:networks} still hold when we allow
oriented cycles with non-zero rotor in this way.

\begin{corollary}
Every $X \in U_{\geq 0}$ which is not totally positive
is representable by a finite cylindric network.
\end{corollary}

\subsection{Extension to the whole formal loop group}
\begin{prop}\label{P:TP}
A matrix $X \in \G_{\geq 0}$ is totally positive if and only if the
matrix $Y \in U_{\geq 0}$ of Theorem \ref{T:red} is totally
positive.
\end{prop}

\begin{lem}\label{L:TP}
Suppose $X \in \G_{\geq 0}$ and $Y \in \G_{>0}$.  Then $XY, YX \in
\G_{>0}$.
\end{lem}
\begin{proof}
By Theorem \ref{T:red}, at least one of the diagonals of $X$ has
only non-zero entries.  The statement follows easily.
\end{proof}

\begin{proof}[Proof of Proposition \ref{P:TP}]
The ``if'' direction follows immediately from Lemma \ref{L:TP}.  For
the other direction, it is enough to show that if $X \in U_{\geq 0}$
is not totally positive, and $Y$ is a finitely supported matrix
(such as $FS^k$ in Lemma \ref{T:red}) then $XY$ is not totally
positive. By Lemma \ref{lem:sol}, all minors of $X$ sufficiently far
from the diagonal vanish.  The statement then follows from the
Cauchy-Binet formula \eqref{E:CB}.
\end{proof}

\begin{thm}\label{T:notTP}
A matrix $X \in \G_{\geq 0}$ is not totally positive if and only if
it is a finite product of whirls, curls, upper or lower Chevalley
generators, and shift matrices.
\end{thm}
\begin{proof}
The ``only if'' direction follows from Proposition \ref{P:TP} and
Theorem \ref{thm:tp}.  For, the ``if'' direction, all stated
generators have all minors sufficiently northeast of the diagonal
vanishing; that is all minors $\Delta_{I,J}$ where $I =
\{i_1,\ldots,i_k\}$, $J = \{j_1,\ldots,j_k\}$ and $i_t \leq j_t - s$
for some $s$.  Thus any finite product of such matrices will have
the same property.
\end{proof}

\begin{example}
We already know that the element of $\G$ in Example \ref{ex:3} is
representable by a cylindric network. We should also be able to
factor it the way it is described in the theorem. Indeed, one can
check that
$$
f_1(1/3) f_2(9/16) T(9/8,16/3) e_1(128/45) e_2(150/368)
M(23/30,5/23)
$$
is one such factorization, where $T$ denotes an element of the
torus.
\end{example}

\begin{cor}
A matrix $X \in \G_{\geq 0}$ is not totally positive if and only if
there exists $s$ and $k$ such that $\Delta_{I,J}(X) = 0$ whenever $I
= \{i_1,\ldots,i_k\}$, $J = \{j_1,\ldots,j_k\}$ satisfy $i_t \leq
j_t - s$.
\end{cor}

\begin{cor}
Suppose $X \in \G_{\geq 0}$.  Then either every sufficiently large
and sufficiently northeast minor of $X$ vanishes, or every
sufficiently northeast minor of $X$ is positive.
\end{cor}

\section{Whirl and curl relations}\label{sec:comm}
This section is concerned with relations that exist between products
of whirls, curls and Chevalley generators.  In the case $n=1$ there
are no Chevalley generators, while whirls and curls simply commute.
For arbitrary $n$, we introduce a relation between products of two
whirls or two curls, and another one between a whirl and a curl.  We
call these relations {\it {commutation relations}}, even though the
factors do not commute.  The commutation relations are well-defined
only when one of the two factors is non-degenerate.  If both factors
are degenerate the commutation relations are not well-defined.
However, in this case we may use the usual braid relations between
Chevalley generators (see \cite{Lu1, LP}).

Let $\a = (a_1,\ldots,a_n), \b = (b_1,\ldots,b_n) \in \R_{\geq 0}^n$
be two sets of parameters.  Define $$\kappa_i({\bf a}, {\bf b}) =
\sum_{j=i}^{i+n-1} \prod_{k=i+1}^j b_{k}
\prod_{k=j+1}^{i+n-1}a_{k}.$$   We call $\a$ degenerate if at least
one of the $a_i$ vanishes.  Let $R \subset \R_{\geq 0}^n \times
\R_{\geq 0}^n$ be the subset of pairs $(\a,\b)$ such that at most
one of $\a$ and $\b$ is degenerate.

Now define a map $\eta: R \to R$ by $\eta(\a,\b) = (\b',\a')$ where
$$b'_i = \frac{b_{i+1} \kappa_{i+1}({\bf a}, {\bf b})}{ \kappa_{i}({\bf a}, {\bf b})} \ \ \ \ \ \ \ \ a'_i = \frac{a_{i-1} \kappa_{i-1}({\bf a}, {\bf b})}{ \kappa_{i}({\bf a}, {\bf b})}.$$
It is not hard to see that $\eta$ is a well-defined map from $R$ to
$R$.  For example, for $n=3$ we have $$b'_1 = \frac{b_2 (a_1 a_3 +
a_1 b_3 + b_1 b_3)}{a_2 a_3 + b_2 a_3 +b_2 b_3}.$$

\begin{lemma} \label{lem:eta}
The function $\eta$ has the following properties:
\begin{enumerate}
  \item $a'_i + b'_i = a_i + b_i$;
  \item $b'_i a'_{i+1} = a_i b_{i+1}$;
  \item $\prod_i a_i = \prod_i a'_i$, $\prod_i b_i = \prod_i b'_i$;
  \item $\eta$ is an involution.
\end{enumerate}
\end{lemma}

\begin{proof}
We have $$(a_i + b_i) \kappa_{i}({\bf a}, {\bf b}) = (a_i + b_i)
\sum_{j=i}^{i+n-1} \prod_{k=i+1}^j b_{k}
\prod_{k=j+1}^{i+n-1}a_{k}$$ $$= a_i  \sum_{j=i+1}^{i+n-1}
\prod_{k=i+1}^j b_{k} \prod_{k=j+1}^{i+n-1}a_{k} + \prod_i a_i + b_i
\sum_{j=i}^{i+n-2} \prod_{k=i+1}^j b_{k} \prod_{k=j+1}^{i+n-1}a_{k}
+ \prod_i b_i$$ $$= b_{i+1}  \sum_{j=i+1}^{i+n-1} \prod_{k=i+2}^j
b_{k} \prod_{k=j+1}^{i+n}a_{k} + \prod_i b_i + a_{i-1}
\sum_{j=i}^{i+n-2} \prod_{k=i}^j b_{k} \prod_{k=j+1}^{i+n-2}a_{k} +
\prod_i a_i$$ $$= a_{i-1} \kappa_{i-1}({\bf a}, {\bf b}) + b_{i+1}
\kappa_{i+1}({\bf a}, {\bf b}),$$ from which (1) follows. (2) and
(3) are straight forward from the definition of $\eta$.

To prove (4), first suppose that $\a$ and $\b$ are both
non-degenerate.  Using (1) and (2), one can solve for $b'_1$ to get
$$b'_1 = a_1+b_1-\frac{a_nb_1}{a_n+b_n-\frac{a_{n-1}b_n}{\ldots
-\frac{a_1b_2}{b'_1}}}.$$ This is a quadratic equation in $b'_1$ and
thus has at most two distinct solutions.  Furthermore, it is clear
that $b'_1$, together with the values of $(a_i + b_i)$ and
$(a_ib_{i+1})$ determine $(\b',\a')$ once (1) and (2) are known.  So
there are at most two solutions to (1) and (2) (with $(a_i + b_i)$
and $(a_ib_{i+1})$ fixed) which are $(\a,\b)$ and $(\b',\a')$.
Applying $\eta$ to $(\b',\a')$ must again give one of these
solutions.  Now we observe that $\eta(\a,\b) = (\a,\b)$ if and
only if $\prod_i a_i = \prod_i b_i$.  It thus follows from (3) that
$\eta(\b',\a')= (\a,\b)$.

Finally, the function $\eta^2(\a,\b)$ is continuous, so the claim
extends to the case that $\a$ or $\b$ is degenerate.
\end{proof}

\begin{theorem}\label{T:commute}
If $\eta({\bf a}, {\bf b}) = ({\bf b'},{\bf a'})$ then $M(\a)M(\b) =
M(\b')M(\a')$ and $N(\b)N(\a)=N(\a')N(\b')$.
\end{theorem}

\begin{proof}
The non-zero entries above diagonal in $M(a_1,\ldots,a_n)
M(b_1,\ldots,b_n)$ are $a_i + b_i$ and $a_i b_{i+1}$. Now apply (1)
and (2) from Lemma \ref{lem:eta}. The case of curls follows by
taking $^{-c}$ of the whirl case.
\end{proof}
\begin{example}
In Examples \ref{ex:4} and \ref{ex:5} we saw that $N(1, 1) N(1, 2) =
N(\frac{4}{3}, \frac{3}{2}) N(\frac{2}{3}, \frac{3}{2})$. Indeed,
let us take ${\bf {a}} = (1,1)$ and ${\bf {b}} = (1,2)$. Then
$\kappa_{1}({\bf a}, {\bf b}) = 1+2 =3$ and $\kappa_{2}({\bf a},
{\bf b}) = 1+1 =2$, which gives $b'_1 = \frac{2 \cdot 2}{3}$, $b'_2
= \frac{1 \cdot 3}{2}$, $a'_1 = \frac{1 \cdot 2}{3}$, $a'_2 =
\frac{1 \cdot 3}{2}$ as desired.
\end{example}

If $(\a^{(1)},\a^{(2)},\ldots,\a^{(k)})$ is a sequence of $n$-tuples
of nonnegative real numbers, we denote by
$\eta_i(\a^{(1)},\a^{(2)},\ldots,\a^{(k)})$ the sequence of
$n$-tuples obtained by applying $\eta$ to $(\a^{(i)},\a^{(i+1)})$
(assuming $\eta$ is well-defined).

\begin{theorem}\label{thm:braid}
The map $\eta$ satisfies the braid relation:
$$
\eta_i \circ \eta_{i+1} \circ
\eta_i(\a^{(1)},\a^{(2)},\ldots,\a^{(k)}) = \eta_{i+1} \circ \eta_i
\circ \eta_{i+1} (\a^{(1)},\a^{(2)},\ldots,\a^{(k)})
$$
whenever the expressions are well-defined.
\end{theorem}

\begin{proof}
We may suppose $k = 3$, and consider a triple $(\a,\b,\c)$.  Since
we are interested in the equality of two rational functions, it
suffices to show that the statement is true for a Zariski dense set.
We consider tuples $({\bf a}, {\bf b}, {\bf c})$ such that $\prod_i
a_i > \prod_i b_i > \prod_i c_i$.  Since this set locally looks like
$\mathbb \R^{3n}$, it is clear that it is Zariski dense.  Let $({\bf
c'}, {\bf b'}, {\bf a'})$ and $({\bf c''}, {\bf b''}, {\bf a''})$ be
the triples on the left and right hand side of the statement of the
theorem.

Then using Lemma \ref{lem:eta}, we deduce $\prod_i a'_i = \prod_i
a''_i = \prod_i a_i$, $\prod_i b'_i = \prod_i b''_i = \prod_i b_i$,
$\prod_i c'_i = \prod_i c''_i = \prod_i c_i$.  Using Theorem
\ref{T:commute}, we have $$X = N(\c)N(\b)N(\a) = N(\a')N(\b')N(\c') =
N(\a'')N(\b'')N(\c'').$$  By assumption we have $r(X) = 1/(\prod_i
a_i)$ (since $r(N(\a)) = \prod_i 1/a_i$), and by Lemma
\ref{lem:limits} and Lemma \ref{lem:ASWconv}, we deduce that $\a' =
\a''$.  Similarly $\b' = \b''$ and $\c' = \c''$.
\end{proof}

\begin{cor} \label{cor:sinf}
The $k-1$ maps $\eta_1,\eta_2,\ldots,\eta_{k-1}$ generate an action
of $S_k$ on $(\R_{>0}^n)^k$.
\end{cor}
\begin{proof}
By Theorem \ref{thm:braid} and Lemma \ref{lem:eta}, the maps satisfy
the relations of the simple generators of the symmetric group $S_k$,
and so generate an action of a subgroup of $S_k$.  But if we pick a
point $(\a^{(1)},\ldots,\a^{(k)}) \in (\R_{>0}^{n})^k$ such that
$\prod_i a^{(1)}_i > \prod_i a^{(2)}_i > \cdots > \prod_i a^{(k)}_i
> 0$ then the orbit of this point under the $k-1$ maps has size at
least $k!$.  Thus the maps generate an action of $S_k$.
\end{proof}

\begin{remark}
Corollary \ref{cor:sinf} had previously been established in a number
of different contexts: by Noumi and Yamada (see \cite{NY}) in the
context of birational actions of affine Weyl groups, by Kirillov
\cite{Ki} in his study of tropical combinatorics, by
Berenstein-Kazhdan \cite{BK} in the theory of geometrical crystals,
and by Etingof \cite{Et} in the study of set-theoretical solutions
of Yang-Baxter equations.
\end{remark}

Now define $\theta: R \to R$ by $\theta({\bf a}, {\bf b})=({\bf
b'},{\bf a'})$ where $$b'_i = \frac{(a_i +
b_i)b_{i+1}}{a_{i+1}+b_{i+1}} \ \ \ \ \ \ a'_i = \frac{(a_i +
b_i)a_{i+1}}{a_{i+1}+b_{i+1}}.$$

\begin{lemma} \label{lem:theta}
The function $\theta$ has the following properties:
\begin{enumerate}
  \item $a_i + b_i = a'_i + b'_i$
  \item $\prod_i a_i = \prod_i a'_i$, $\prod_i b_i = \prod_i b'_i$
  \item the inverse map $\theta^{-1}$ is given by $({\bf a}, {\bf b}) \mapsto ({\bf b'},{\bf a'})$ where
  $$
  b'_i = \frac{(a_i + b_i)b_{i-1}}{a_{i-1}+b_{i-1}} \ \ \ \ \ \
  a'_i = \frac{(a_i + b_i)a_{i-1}}{a_{i-1}+b_{i-1}}
  $$
  \item if $\theta^{2k-1}({\bf a}, {\bf b})=({\bf b'},{\bf a'})$ then
  $$b'_i = \frac{(a_i+b_i)b_{i+2k-1}}{a_{i+2k-1}+b_{i+2k-1}} \ \ \ \ \ \
  a'_i = \frac{(a_i+b_i)a_{i+2k-1}}{a_{i+2k-1}+b_{i+2k-1}}$$
  \item if $\theta^{2k}({\bf a}, {\bf b})=({\bf a'},{\bf b'})$ then
  $$a'_i = \frac{(a_i+b_i)a_{i+2k}}{a_{i+2k}+b_{i+2k}} \ \ \ \ \ \
  b'_i = \frac{(a_i+b_i)b_{i+2k}}{a_{i+2k}+b_{i+2k}}$$
  \item $\theta^{{\rm lcm}(n,2)}$ is the identity map.
\end{enumerate}
\end{lemma}

\begin{proof}
Statements (1), (2) and (3) follow directly from definition, (4) and
(5) are easily verified by induction, (6) follows from (4) and (5).
\end{proof}

\begin{theorem}
If $\theta({\bf a}, {\bf b})=({\bf b'},{\bf a'})$ then $M(\a) N(\b)
= N(\b') M(\a')$.
\end{theorem}

\begin{proof}
The $(i,j)$-th entry in $N(b'_1,\ldots,b'_n) M(a'_1,\ldots,a'_n)$ is
$$(b'_{j-1}+a'_{j-1}) \prod_{k=i}^{j-2} b'_k = (a_{j-1}+b_{j-1})
\prod_{k=i}^{j-2} \frac{(a_k + b_k)b_{k+1}}{a_{k+1}+b_{k+1}}=
(a_{i}+b_{i}) \prod_{k=i+1}^{j-1} b_{k}$$ which is exactly the
$(i,j)$-th entry of $M(a_1,\ldots,a_n) N(b_1,\ldots,b_n)$.
\end{proof}

Both $\eta$ and $\theta$ are well-defined as long as at least one of
${\bf a}$ and ${\bf b}$ is non-degenerate.  The following lemma
shows that interpreting a Chevalley generator as a degenerate whirl
and using $\eta$ results in the same relation as interpreting a
Chevalley generator as a degenerate curl and using $\theta$.

\begin{lemma} \label{lem:chevcom}
We have $\eta((0, \ldots, 0, a_i, 0, \ldots, 0), {\bf b}) =
\theta^{-1}((0, \ldots, 0, a_i, 0, \ldots, 0), {\bf b}) = ({\bf
b'},{\bf a'})$ and the map can be described as follows:
\begin{enumerate}
 \item $b'_k = b_k$, $k \not = i+1, i$;
 \item $b'_{i+1}=\frac{b_{i+1}b_i}{a_i+b_i}$, $a'_i=a_i+b_i$;
 \item $a'_k = 0$, $k \not = i+1$;
 \item $a'_{i+1}=\frac{b_{i+1}a_i}{a_i+b_i}$.
\end{enumerate}
\end{lemma}

\begin{proof}
A direct computation. For example, one has $\kappa_{i+2}({\bf a},
{\bf b}) = \prod_{k \not = i+2} b_k$, $\kappa_{i+1}({\bf a}, {\bf
b}) = (a_i+b_i) \prod_{k \not = i,i+1} b_k$ and by definition
$b'_{i+1} = b_{i+2} \kappa_{i+2}({\bf a}, {\bf b})/
\kappa_{i+1}({\bf a}, {\bf b}) = b_{i+1}b_i/(a_i+b_i)$.
\end{proof}

For later use we also give the following result.

\begin{lemma} \label{lem:chevcom2}
The map $\theta((0, \ldots, 0, a_i, 0, \ldots, 0), {\bf b}) = ({\bf
b'},{\bf a'})$ can be described as follows:
\begin{enumerate}
 \item $b'_k = b_k$, $k \not = i-1, i$;
 \item $b'_{i-1}=\frac{b_{i-1}b_i}{a_i+b_i}$, $a'_i=a_i+b_i$;
 \item $a'_k = 0$, $k \not = i-1$;
 \item $a'_{i-1}=\frac{b_{i-1}a_i}{a_i+b_i}$.
\end{enumerate}
\end{lemma}

\begin{proof}
Direct computation from the definitions.
\end{proof}

\section{Infinite products of whirls and curls}\label{sec:infprod}
\subsection{Infinite whirls and curls}
For a possibly infinite sequence of matrices $(X^{(i)})_{i =
1}^{\infty}$ we write $\prod_{i=1}^{\infty} X^{(i)}$ for the limit
$$
\lim_{k \to \infty} (X^{(1)} X^{(2)} \cdots X^{(k)}).
$$
Similarly define $\prod_{i = -\infty}^{-1} X^{(i)}$ by
$$
\lim_{k \to -\infty} (X^{(k)} X^{(k+1)} \cdots X^{(-1)}).
$$

\begin{lem}\label{lem:infinitewhirl}
Let $(a^{(1)}_{1},a^{(1)}_{2},\ldots,a^{(1)}_{n}),
(a^{(2)}_{1},a^{(2)}_{2},\ldots,a^{(2)}_{n}), \ldots$ be an infinite sequence of
$n$-tuples of nonnegative numbers such that $\sum_{i = 1}^\infty
\sum_{j = 1}^n a^{(i)}_{j} < \infty$.  Then the limits
$$
\prod_{i=1}^{\infty} M(a^{(i)}_{1}\ldots,a^{(i)}_{n}),
\prod_{i=-\infty}^{-1} M(a^{(-i)}_{1}\ldots,a^{(-i)}_{n}),
\prod_{i=1}^{\infty} N(a^{(i)}_{1}\ldots,a^{(i)}_{n}),
\prod_{i=-\infty}^{-1} N(a^{(-i)}_{1}\ldots,a^{(-i)}_{n})
$$
exist and are TNN matrices.  Conversely, the limits exist only if
the sum is finite.
\end{lem}

\begin{proof}
We will prove the statement for $\prod_{i=1}^{\infty}
M(a^{(i)}_{1}\ldots,a^{(i)}_{n})$.  The result for curls is
obtained by taking inverses.  Each entry of the sequence
$\prod_{i=1}^{k} M(a^{(i)}_{1}\ldots,a^{(i)}_{n})$ is non-decreasing as $k
\to \infty$ so it suffices to prove that every entry is bounded.  It
is easy to see that the entries directly above the diagonal are
bounded by $\alpha = \sum_{i = 1}^\infty \sum_{j = 1}^n a^{(i)}_{j}$. By
induction, one sees that entries along the $d$-th diagonal are
bounded by $\alpha^d$.

By Lemma \ref{lem:limits} we see that $\prod_{i=1}^{\infty}
M(a^{(i)}_{1}\ldots,a^{(i)}_{n})$ is TNN.
\end{proof}

We call the products above {\it {right-infinite whirls}}, {\it
{left-infinite whirls}}, {\it {right-infinite curls}} and {\it
{left-infinite curls}}. If $X$ is an infinite whirl (resp.~curl) we
say that $X$ is of {\it whirl type} (resp.~{\it curl type}).

\begin{lem}\label{lem:infinitewhirldet}
Let $X$ one of the infinite products in Lemma
\ref{lem:infinitewhirl}.  Then the folded determinant of $X$ is
given by
$$
\det(\overline{X}(t)) = \begin{cases} \prod_{i = 1}^\infty (1 + (-1)^{n+1} (\prod_{j=1}^n a^{(i)}_{j})\, t) &\mbox{if $X$ is of whirl type} \\
\prod_{i=1}^{\infty} \frac{1}{1 - (\prod_{j=1}^n a^{(i)}_{j})\, t}
&\mbox{if $X$ is of curl type}.
\end{cases}
$$
\end{lem}
\begin{proof}
Each coefficient of $\det(\overline{X}(t))$ depends on only finitely
many entries of $X$.  The statement then follows from taking an
infinite product of Lemma \ref{lem:whirldet}.
\end{proof}

\subsection{Loop symmetric functions}

In this subsection, we assume familiarity with the theory of Young
tableaux and symmetric functions \cite{EC2}.  Let $Y =
(y_{k,l})_{k,l=-\infty}^\infty = \prod_{i=1}^\infty N(\xx_i)$
 be a right-infinite curl, where $\x_i =
 (x_i^{(1)},x_i^{(2)},\ldots,x_i^{(n)})$.  Note that in order to
agree with usual symmetric function conventions, we have labeled (in
this subsection only) the upper and lower indices of the curl
parameters $x_i^{(j)}$ in the opposite manner to our usual notation.
We caution the reader that with variables $a_j^{(i)}$ it is the
lower index that is in $\Z/n\Z$.

We now interpret the entries of $Y$ as analogs of homogeneous
symmetric functions in variables $x_i^{(j)}$.   Define for each $r
\geq 1$ and each $k \in \Z/n\Z$,
$$
h_{r}^{(k)}({\xx}) = \sum_I x^{(k)}_{i_1} x^{(k+1)}_{i_2} \dotsc
x^{(k+r-1)}_{i_{r}}
$$
where the sum is taken over all weakly increasing sequences $1 \leq
i_1 \leq i_2 \leq \ldots \leq i_{r}$.  We shall call the
$h_r^{(k)}({\xx})$ {\it loop homogeneous symmetric functions}.

\begin{lemma}\label{L:homog}
Let $Y = \prod_{i=1}^\infty N(\x_{i})$.  Then we have $y_{k,l} =
h_{l-k}^{(k)}({\xx})$.
\end{lemma}

\begin{proof}
We first argue the statement is valid for any finite number of
curls. We proceed by induction, the case of one curl follows
trivially from the definition of curls. Assume we have already shown
that the entries of $Y = \prod_{i = 1}^{m-1} N(x_i^{(1)}, \ldots,
x_i^{(n)})$ are described by the stated formula. Let us consider $Y'
= Y N$ where $N = N(x_m^{(1)}, \ldots, x_m^{(n)})$. We have
$y'_{k,l} = \sum_{t=0}^{l-k} y_{k,k+t} N_{k+t, l}$. We know that
$y_{k,k+t}$ equals the sum $\sum_I x_{i_1}^{(k)} x_{i_2}^{(k+1)}
\dotsc x^{(k+t-1)}_{i_{t}}$ over all weakly increasing sequences $I$
of length $t$.  At the same time $N_{k+t, l}$ equals the product
$x_{m}^{(k+t)} x_{m}^{(k+t+1)} \dotsc x_m^{(l)}$. Thus the term
$y_{k,k+t} N_{k+t, l}$ of the summation equals the sum $ \sum_{I'}
x_{i_1}^{(k)} x_{i_2}^{(k+1)} \dotsc x_{i_{l-k}}^{(l-1)}$ over all
sequences
$$I' = i_1 \leq i_2 \leq \ldots \leq i_{k+t-1} < i_{k+t} = \ldots =
i_{l} = m.$$ Summing over $t$ gives the desired result.

For an infinite product of curls, the result follows from taking the
limit $m \to \infty$.  The limit exists by Lemma
\ref{lem:infinitewhirl}.
\end{proof}

Now we provide an analog of Jacobi-Trudi formula, giving an
interpretation for minors of $Y$ as generalizations
$s_\lambda({\xx})$ of skew Schur functions, which we call {\it loop
Schur functions}.  Let $\lambda = \rho/\nu$ be a skew shape, which
we shall draw in the English notation:
\begin{figure}[h!]     \begin{center}
    \input{loo7.pstex_t}
    \end{center}
    \caption{}
    \label{fig:loo7}
\end{figure}

A square $s = (i,j)$ in the $i$-th row and $j$-th column has {\it
content} $j - i$ and has {\it residue} $r(s) = \overline{j- i} \in
\Z/n\Z$. Recall that a semistandard Young tableaux $T$ with shape
$\lambda$ is a filling of each square $s \in \lambda$ with an
integer $T(s) \in \Z_{> 0}$ so that the rows are weakly-increasing,
and columns are increasing.  An example of a semistandard tableau is
given on the right in Figure \ref{fig:loo7}. The weight $x^T$ of a
tableaux $T$ is given by $x^T = \prod_{s \in \lambda}
x_{T(s)}^{(r(s))}$.  We define the loop Schur function by
$$
s_\lambda({\x}) = \sum_{T} x^T
$$
where the summation is over all semistandard Young tableaux of
(skew) shape $\lambda$.  We shall also need several alternative
definitions. We define the {\it {mirror residue}} $\r(s) =
\overline{i- j} \in \Z/n\Z$. We define $\x^T = \prod_{s \in \lambda}
x_{T(s)}^{(\r(s))}$ and the {\it mirror loop Schur functions}
$$
\s_\lambda({\xx}) = \sum_{T} \x^T.
$$

\begin{theorem}\label{T:schur}
Let $Y = \prod_{i=1}^\infty N(\xx_{i})$.  Let $I = i_1 < i_2 <
\ldots < i_k$ and $J = j_1 < j_2 < \ldots < j_k$ be two sequences of
integers such that $i_t \leq j_t$.  Define $$\lambda = \lambda(I,J)
= (j_k, j_{k-1} + 1, \ldots, j_1+k-1)/(i_k,i_{k-1}+ 1, \ldots,
i_1+k-1).$$ Then
$$
\Delta_{I,J}(Y) = \det(h_{j_t - i_s}^{(i_s)}(\xx))_{s,t=1}^k =
s_\lambda({\xx}).
$$
\end{theorem}
Note that if $I$ and $J$ do not satisfy the condition $i_t \leq j_t$
then $\Delta_{I,J}(Y) = 0$.

\begin{proof}
The first equality follows from Lemma \ref{L:homog}.  We prove the
second inequality using the Gessel-Viennot method in the standard
manner.  We refer the reader to \cite[Chapter 7]{EC2} for details
concerning this method.

Consider the square lattice grid in the plane, and orient all
vertical edges north and all horizontal edges east.  Assign to
vertical edges weight $1$. Assign to a horizontal edge of the grid
connecting $(p,q)$ with $(p+1, q)$ the weight $x_{q+1}^{(p)}$.
Consider $k$ sources with coordinates $(i_s, 0)$, $s = 1, \ldots, k$
and $k$ sinks with coordinates $(j_t, \infty)$, $t = 1, \ldots, k$.
One checks directly that the weight generating function of paths
from the source $(i_s,0)$ to $(j_t,\infty)$ is equal to $h_{j_t -
i_s}^{(i_s)}$.  By the Gessel-Viennot method, the determinant
$\det(h_{j_t - i_s}^{(i_s)})_{s,t=1}^k$ is the weight generating
function of non-intersecting families of paths from these $k$
sources to the $k$-sinks.  It is easy to see that such families are
in bijection with semistandard tableaux $T$ of shape $\lambda$, and
that the weight of the path family corresponding to a tableau $T$ is
exactly $x^T$.
\end{proof}

\begin{example}
Let $n = 3$.  For $I = (1,2,5)$ and $J = (4,7,9)$ we get the skew
shape shown in Figure \ref{fig:loo7}. The monomial corresponding to
the shown semistandard filling is $$x_1^{(1)} x_2^{(1)} x_3^{(1)}
(x_1^{(2)})^2 (x_3^{(2)})^3 x_1^{(3)} x_2^{(3)} (x_4^{(3)})^2.$$
\end{example}

We now state similar theorems for right-infinite whirls, and the
proofs are completely analogous.  Let $Y = \prod_{i \geq 1}
M(x_i^{(1)}, \ldots, x_i^{(n)})$ be a right-infinite whirl. We
define the {\it (mirror) loop elementary symmetric functions}
$\e_{r}^{(k)}({\xx}) = \sum_I x_{i_1}^{(k)} x_{i_2}^{(k+1)} \dotsc
x_{i_r}^{(k+r-1)}$, where the sum is taken over all increasing
sequences $i_1 < i_2 < \ldots < i_{r}$.

\begin{lemma}\label{L:elem}
Let $Y = \prod_{i = 1}^\infty M(\xx_{i})$.  We have $y_{k,l} =
\e_{l-k}^{(k)}({\xx})$.
\end{lemma}

If $\lambda$ is a skew shape, we let $\lambda'$ denote the conjugate
of $\lambda$, obtained reflecting $\lambda$ in the main diagonal.

\begin{theorem}
Let $Y = \prod_{i = 1}^\infty M(\xx_i)$.  Let $I = i_1 < i_2 <
\ldots < i_k$ and $J = j_1 < j_2 < \ldots < j_k$ be two sequences of
integers such that $i_t \leq j_t$.   Define
$$\lambda = \lambda(I,J) = (j_k, j_{k-1} + 1, \ldots,
j_1+k-1)/(i_k,i_{k-1}+ 1 \ldots, i_1+k-1).$$ Then
$$
\Delta_{I,J}(Y) = \det(\e_{j_t - i_s}^{(i_s)}(\xx))_{s,t=1}^k =
\s_{\lambda'}({\xx}).
$$
\end{theorem}

\begin{remark}
\def\LSym{{\rm LSym}}
If we consider the $x^{(j)}_i$ as variables, then $\{h_r^{(k)}\}$
are algebraically independent (and so are the $\{\e_r^{(k)}\}$). The
commutative ring which the $\{h_r^{(k)}\}$ generate we call {\it
loop symmetric functions}, denoted $\LSym$.  (The ring generated by
the $\{\e_r^{(k)}\}$ is distinct from $\LSym$, considered as
subrings of the ring of formal power series.) The ring $\LSym$ is a
Hopf algebra which coincides with the usual ring of symmetric
functions when $n = 1$. We shall study $\LSym$ in detail in future
work.
\end{remark}

\begin{remark}
Our loop homogeneous symmetric functions also appear in the context
of Noumi-Yamada's study of discrete Painlev\'{e} dynamical systems,
see \cite{Y}.
\end{remark}

\begin{remark}
The concept of {\it {chess tableaux}} in the work of Scott \cite{Sc}
seems to be related to the weight of the tableaux as defined here.
\end{remark}

\subsection{Basic properties of infinite whirls and curls}
We say that a matrix $A = A(t)$ is entire if every entry of $A$ is
entire.  We say $X \in U$ is entire if $A(X)$ is.

\begin{lem}\label{lem:whirlmu}
Let $X = \prod_{i=1}^{\infty} M(\a^{(i)})$ (resp.~$X =
\prod_{i=-\infty}^{-1} M(\a^{(-i)})$) be well-defined as in Lemma
\ref{lem:infinitewhirl} and not finitely supported. Then $\mu_i(X) =
0$ (resp.~$\epsilon_i(X) = 0$) for each $i$. In particular, $X$ is
entire.
\end{lem}

We remind the reader that with the $\a$ variables, the lower index
is the one taking values in $\Z/n\Z$.

\begin{proof}
Let us consider $X = \prod_{i=1}^{\infty} M(\a^{(i)})$; the other
case is similar.  Using Lemma \ref{L:elem} and the definition of
$\mu_i(X)$, we must show for each $k$ that the ratio
$\e^{(k+s)}_{s+1}({\bf a})/\e^{(k+s)}_{s}({\bf a})$ approaches 0 as $s
\to \infty$. We know that $\e^{(k+s)}_s({\bf a})$ is the generating
function of semistandard tableau with shape a column of size $s$ and
initial residue $k+s$.  Given such a column tableau $T$ with size
$s+1$ we may produce a column tableau $T'$ with size $s$ by removing
the letter in the last (lowest) box.  A fixed column tableau $T'$
with size $s$ can be obtained in this way for each possible value of
the last box.
 But for sufficiently large $s$, we have $\sum_{i \geq s}^\infty \sum_{j=1}^n a^{(i)}_{j}
< \varepsilon$, for any given $\varepsilon > 0$.  Thus for sufficiently large $s$, we have
$\e^{(k+s)}_{s+1}({\bf a})/\e^{(k+s)}_{s}({\bf a}) <\varepsilon$, as required.
\end{proof}

\begin{lem}\label{L:aep}
Let $X = \prod_{i=1}^{\infty} N(\a^{(i)})$ or $X =
\prod_{i=-\infty}^{-1} N(\a^{(-i)})$ be well-defined as in Lemma
\ref{lem:infinitewhirl}.  Define $b_i = \prod_{j=1}^n a_j^{(i)}$ and
assume that $b_1 = \max_i b_i \neq 0$. Then $\epsilon_j(X) =
a^{(1)}_j$.
\end{lem}
\begin{proof}
We consider the case $X = \prod_{i=1}^{\infty} N(\a^{(i)})$.  Let
$\alpha = a^{(1)}_j$.  The sum of all $a^{(i)}_{j}$ converges, so
certainly $b_i \to 0$ and the maximum $b = \max_i b_i$ exists.  By
Lemma \ref{lem:ASWconv}, we have $\epsilon_j(X) \geq \alpha$.  By
Lemmata \ref{L:homog} and \ref{lem:limits}, it suffices to check
that $\lim_{s \to \infty} h^{(j)}_{s+1}({\bf a})/h^{(j+1)}_s({\bf
a}) \leq \alpha$.  Let $S_{s+1}$ be the set of semistandard tableaux
of shape a row of length $s+1$, shifted in the plane so that the
initial box has residue $j$. Similarly let $S_{s}$ be the set of
semistandard tableaux of shape a row of length $s$, with initial box
having residue $j+1$.  If $S$ is a set of tableaux, then we write
$\wt(S) = \sum_{T \in S} a^T$. Thus $\wt(S_{s+1}) =
h^{(j)}_{s+1}({\bf a})$ and $\wt(S_{s}) = h^{(j+1)}_{s}({\bf a})$,
so it suffices to prove that for sufficiently large $s$ we have
$\wt(S_{s+1}) \leq (a^{(1)}_j+ \varepsilon) \wt(S_{s})$ for
arbitrarily small $\varepsilon$. Given a tableau $T \in S_{s}$ we
can obtain a tableau $T' \in S_{s+1}$ by adding the number 1 in
front, and we have $a^{T'} = \alpha \cdot a^T$.  Let $S'_{s} \subset
S_{s}$ be the subset of tableaux which start with a number 2 or
greater, and let $S^*_s = S_s - S'_s$.  It is enough to show that
for sufficiently large $s$ we have $\wt(S'_s) \leq \varepsilon
\wt(S^*_s)$ for arbitrarily small $\varepsilon$.  (For tableaux $T
\in S^*_s$ only the number 1 can be added in front, and every $T'
\in S_{s+1}$ is obtained by adding some number in front of some $T
\in S_s$.)

Pick $R$ so that $$\sum_{i \geq R} \sum_{j=1}^n a_j^{(i)} <
\min(a_1^{(1)},a_2^{(1)},\ldots,a_n^{(1)}).$$ This can be done since
the sum $\sum_{i,j} a_j^{(i)}$ is finite.  Let $W \subset S'_s$
denote the tableaux labeled with numbers from $\{2,3,\ldots,R\}$,
where we now declare that for $T \in W$, the tableau has a modified
weight $\wt'$: the number $R$ in a square with residue $j$ has
weight $a_j^{(1)}$.  By the construction of $R$, we deduce that
$\wt(S'_s) \leq \wt'(W)$ using this modified weight.

Pick $s > nR^2/\varepsilon$.  Given a tableau $T \in W$ there are at
least $s/R$ (consecutive) numbers all equal to some $r \in
[2,R]$. We pick the smallest such $r$.  We define a collection
$\gamma(T) \subset S^*_s$ by removing the first $n, 2n ,\ldots,$ of
these numbers from $T$, and replacing them with 1's in the beginning
of $T$. Thus $\gamma(T)$ consists of at least $s/nR$ distinct
tableaux. Furthermore, each tableau in $\gamma(T)$ has weight
greater than the (modified) weight of $T$, and each tableau in
$S^*_s$ can occur this way in at most $R$ ways.  We conclude that
$$
(s/nR) \wt(S'_s) \leq (s/nR) \wt'(W) \leq R \; \wt(S^*_s)
$$
so that $\wt(S'_s) \leq \varepsilon \wt(S^*_s)$, as required.
\end{proof}

\begin{lem}\label{lem:inverseradius}
Let $X = \prod_{i=1}^{\infty} N(\a^{(i)})$ or $X =
\prod_{i=-\infty}^{-1} N(\a^{(-i)})$ be well-defined as in Lemma
\ref{lem:infinitewhirl}.  Define $b_i = \prod_{j=1}^n a_j^{(i)}$ and
assume that $\max_i b_i \neq 0$.  Then $r(X) = 1/(\max_i b_i)$.
\end{lem}

\begin{proof}
We prove the statement for $X = \prod_{i=1}^{\infty} N(\a^{(i)})$.
Using Theorem \ref{T:commute} possibly repeatedly, we may assume
that $b = b_1$ is maximal.  The result then follows from Lemmata
\ref{lem:limits} and \ref{L:aep}.
\end{proof}

\begin{remark}
The assumption $\max_i b_i \neq 0$ in Lemmata \ref{L:aep} and
\ref{lem:inverseradius} can be removed (see \cite{LP}).
\end{remark}

\begin{cor}
Suppose $X$ is of curl type.  Then the radius of
convergence of $\det(\overline{X})$ is equal to $r(X)$.
\end{cor}

\section{Canonical form}\label{sec:canon}
Let $RC \subset U_{\geq 0}$ denote the set of matrices of the form
$Z = \prod_{i=1}^{\infty} N(a^{(i)}_{1},\ldots,a^{(i)}_{n})$ where
all the $a^{(i)}_{j}$ are strictly positive and the sum of the
$a^{(i)}_j$ converges. In other words, $RC$ is the set of
right-infinite products of non-degenerate curls. Define
$\overline{RC}$ to be the union of $RC$ and the set of finite
products of non-degenerate curls.  Similarly we define $LC$ and
$\overline{LC}$ (left-infinite non-degenerate curls), $RW$ and
$\overline{RW}$ (right-infinite non-degenerate whirls), and $LW$ and
$\overline{LW}$ (left-infinite non-degenerate whirls).

\subsection{Whirl and curl components}
Let $X \in U_{\geq 0}$.  A {\it curl factorization} of $X$ is a
factorization of the form $X = Z \, Y$, where $Y$ is entire and $Z
\in \overline{RC}$.

We say that a (possibly finite) sequence $X = X^{(0)}, X^{(1)},
\ldots$ of TNN matrices is a {\it curl reduction} of $X$
if
\begin{enumerate}
\item
$X^{(i)} = N(a^{(i)}_{1},\ldots,a^{(i)}_{n}) X^{(i+1)}$ for each
$i$, where $N(a^{(i)}_{1},\ldots,a^{(i)}_{n})$ is a non-degenerate
curl.
\item
The limit $Y = \lim_{i \to \infty} X^{(i)}$ is entire.
\end{enumerate}
Note that if (1) holds, the limit $Y$ always exists.  That is
because for fixed $k,l$, the entries $x^{(i)}_{k,l}$ of $X^{(i)}$
are non-increasing, but nonnegative.  (This is the case even if we
allow degenerate curls).  It is clear that curl reductions give rise
to curl factorizations.

\begin{lem}\label{lem:ASWreduction}
Let $X \in U_{\geq 0}$.  Then a curl reduction of $X$ exists.
\end{lem}
\begin{proof}
Define $X^{(k+1)}$ be applying Lemma \ref{lem:ASW} to $X^{(k)}$. Let
$Y = \lim_{k \to \infty} X^{(k)}$. If $Y$ is finitely supported it
is clear that $Y$ is entire, so we assume otherwise, using Lemma
\ref{lem:limits} implicitly in the following. In particular, we
assume that the sequence $X^{(k)}$ involves infinitely many
non-trivial applications of Lemma \ref{lem:ASW}.

We now argue that $$\frac{x_{i,j}^{(k)}}{x_{i+n,j}^{(k)}} \geq
\frac{x_{i,j}^{(k+1)}}{x_{i+n,j}^{(k+1)}}.$$  Indeed, after the
substitution $x_{i,j}^{(k+1)} = x_{i,j}^{(k)} - \epsilon_i
x_{i+1,j}^{(k)}$ and similarly for $x_{i+n,j}^{(k)}$, the above
inequality follows from nonnegativity of the minor of $X^{(k)}$ with
rows $i,i+1$ and columns $j-n, j$.  We have used $\epsilon_i \neq 0$
for this calculation.  We conclude that $\frac{y_{i,j}}{y_{i+n,j}}
\leq \frac{x_{i,j}^{(k)}}{x_{i+n,j}^{(k)}}$ for any $k$. Taking the
limit $j \longrightarrow \infty$ we see that $\prod_{i=1}^n
\epsilon_i(Y) \leq \prod_{i=1}^n \epsilon_i(X^{(k)})$. We know that
the sequence $\frac{1}{r(X^{(k)})} = \prod_{i=1}^n
\epsilon_i(X^{(k)})$ is non-increasing as $k \longrightarrow
\infty$, but stays nonnegative. Assume its limit $\delta$ is
non-zero. Then for each $k$ at least one of the
$\epsilon_i(X^{(k)})$ is not less than $\delta^{1/n}$. This however
would mean that $\sum_{k} \sum_{j=1}^n a^{(k)}_{j}$ diverges, which
is impossible. Thus $\delta = 0$. Since $\prod_{i=1}^n
\epsilon_i(Y)$ is bounded from above by a sequence with zero limit
and is nonnegative, it must be the case that $\prod_{i=1}^n
\epsilon_i(Y)=0$. This is equivalent to $Y$ being entire.
\end{proof}

We denote by $Z(X)= \prod_{i=1}^{\infty}
N(a^{(i)}_{1},\ldots,a^{(i)}_{n})$ the infinite product obtained
from the curl reduction of Lemma \ref{lem:ASWreduction}.  Such
product expressions are called {\it ASW factorizations} of $Z(X)$.

\begin{prop}\label{prop:inversewhirlfact}
Let $X \in U_{\geq 0}$.  Then $X$ has a unique curl factorization.
\end{prop}
\begin{proof}
Suppose $X = ZY$ is some curl factorization of $X$. Let us fix a
factorization of $Z$ as an infinite product of curls.
 Since $Y$ is entire, we have $r(Z) \leq r(X)$. Let
$N(a_1,a_2,\ldots,a_n)$ be the curl factor in the
factorization $Z$ with the smallest radius of convergence, that is,
largest value of $\prod_j a_j$.  By Lemma \ref{lem:inverseradius},
we have $r(Z) = 1/(\prod_j a_j)$. Using the whirl commutation
relations, we may move such a
factor to the front of $Z$, so that $Z = N(a'_1,a'_2,\ldots,a'_n)Z'$
and $\prod_j a'_j = \prod_j a_j$.  By Lemma \ref{lem:ASWconv}, we
have $a'_i \leq \epsilon_i(X)$ so that $r(Z) \geq r(X)$.  It follows
that $r(Z) = r(X)$ and $a'_i = \epsilon_i(X)$.

Repeating this argument, we see that the multiset of radii of
convergence of curls in $Z$ coincides with that of $Z(X)$. Let
$Z^{(k)}$ be the product of the first $k$ curls in the curl
reduction of Lemma \ref{lem:ASWreduction}, so that $Z(X) = \lim_{k
\to \infty} Z^{(k)}$.  It is clear that entry-wise $Z^{(k)}$ is less
than $Z$. Let $N$ be arbitrary. Let $b = \min\{\prod_j a^{(i)}_{j}
\mid i \in [1,N)\}$. Pick $k$ so that $Z^{(k)}$ contains all factors
in $Z(X)$ with radii of convergence less than or equal to $1/b$. Let
$Z'$ be the product of the first $N$ factors in $Z$.  By the whirl
commutation relations (Theorem \ref{T:commute}) we can write
$Z^{(k)} = Z' W$ for some $W \in U_{\geq 0}$ -- $W$ is obtained by
moving to the left all of the factors in $Z$ outside of $Z'$ but
with radius of convergence less than or equal to $1/b$.  The entries
of $Z^{(k)}$ are thus greater than those of $Z'$. It follows that
$Z$ is the limit of the $Z^{(k)}$.
\end{proof}

\begin{remark}
In \cite{BFZ} the following question is posed: explicitly describe
the transition map between two different factorizations of a totally
positive element of $GL_n(\R)$.  Distinct factorizations of totally
positive elements correspond to different double wiring diagrams
\cite{FZ2}. Later it was realized \cite{FZ3} that the graph
connecting different parametrizations can be completed to a regular
graph that is the exchange graph of the corresponding {\it {cluster
algebra}}. It is natural to ask a similar question in our setting.
Let us restrict our attention to infinite products of curls (or
whirls). By Proposition \ref{prop:inversewhirlfact} we have the
distinguished ASW factorization, and any other factorization is
obtainable by the repeated application of whirl commutation
relations. By Corollary \ref{cor:sinf} we can conclude that the
graph describing the adjacency between distinct parametrizations of
an infinite curl is just the Cayley graph of $S_{\infty}$ with
adjacent transpositions as generators.  This graph is already
regular and it seems unlikely that analogues of {\it {non-Pl\"ucker
cluster variables}} could arise.  The situation becomes more subtle
when we allow Chevalley generators in the factorizations.  We plan
to address these questions in \cite{LP}.
\end{remark}

We call $X \in U$ {\it doubly entire} if both $X$ and $X^{-1}$ are
entire.  For TNN matrices, we will usually check the equivalent
condition that $X$ and $X^{-c}$ are entire.

\begin{thm} \label{thm:canon}
Let $X \in U_{\geq 0}$.  Then it has a unique factorization of the
form
$$
X=\prod_{i=1}^{\infty} N(a^{(i)}_{1},\ldots,a^{(i)}_{n}) Y
\prod_{i=-\infty}^{-1} M(b^{(i)}_{1},\ldots,b^{(i)}_{n}),$$ where
all whirls and curls are either non-degenerate or the identity
matrix, and the parameters satisfy $\sum_{i,j} a^{(i)}_{i} + \sum_i
b^{(i)}_{j} < \infty$
and $Y \in U_{\geq 0}$ is doubly entire.
\end{thm}
\begin{proof}
For existence, first use Proposition \ref{prop:inversewhirlfact} to
write $X = Z(X) X'$ where $X'$ is entire.  Now apply Proposition
\ref{prop:inversewhirlfact} to $(X')^{-c}$ to obtain $X' = Y W(X)$
where $Y^{-c}$ is entire and $W(X) = \prod_{i=-\infty}^{-1}
M(b^{(i)}_{1},\ldots,b^{(i)}_{n})$ and all parameters are positive.
We claim that $Y$ is entire.  For otherwise, by Lemma \ref{lem:ASW}
we can write $X' = N(\epsilon_1(Y),\ldots,\epsilon_n(Y))Y'W(X)$
where $Y'W(X)$ is TNN and the $\epsilon_i$ are strictly positive.
But $X'$ is entire so this is impossible by Lemma \ref{lem:ASWconv}.
Thus $X = Z(X) Y W(X)$ is the desired factorization.

For uniqueness, suppose we have a factorization $X = Z Y W$ as in
the statement of the theorem.  By Lemma \ref{lem:whirlmu}, we
may apply Proposition \ref{prop:inversewhirlfact} to $X = Z (YW)$ to
see that $Z = Z(X)$.  Repeating the argument for $X^{-c}$ we see
that $W = W(X)$ is well-defined.  (In particular, $W(X)$ can be
calculated before or after factoring $Z(X)$ out.)
\end{proof}

We call the expression $X = ZYW$ of Theorem \ref{thm:canon} the {\it
canonical form} of $X$.  We call $Z$ the {\it curl component} of $X$
and $W$ the {\it whirl component} of $X$.

\subsection{Doubly entire matrices as exponentials}

\begin{lem}
Suppose $A(t)$ is doubly entire.  Then $A(t) = e^{B(t)}$ for some
entire matrix $B(t)$.
\end{lem}
\begin{proof}
Define $Z(t) = e^{-\int A^{-1}(t) A'(t) dt}$ where $A'(t)$ denotes
$\frac{d}{dt}(A(t))$.  Clearly, $Z(t)$ is an entire matrix.  We may
pick the constant of integration so that $Z(0) = A^{-1}(0)$.  This
is possible because $A^{-1}(0)$ is non-singular (with inverse
$A(0)$). However,
$$\frac{d}{dt}(A(t)Z(t)) = A'(t)Z(t) -
A(t) A^{-1}(t)A'(t)Z(t) = 0.$$ Thus $A(t)Z(t)$ is a constant matrix.
But $A(0)Z(0)$ is the identity matrix, so the result holds with
$B(t) = \int A^{-1}(t) A'(t) dt$ which is clearly entire.
\end{proof}


\subsection{Infinite products of Chevalley
generators}\label{ssec:chev} A product of infinitely many
non-degenerate whirls (resp.~non-degenerate curls) can never be
written as a finite product of non-degenerate whirls
(resp.~non-degenerate curls).  This follows from either Lemma
\ref{lem:infinitewhirldet} or the observation that an infinite
product of non-degenerate whirls must have infinite support.  The
situation for Chevalley generators is markedly different.  For
example, with $n = 2$, one has $ \prod_{i=1}^{\infty} M(a_i,0) =
M(\sum_{i=1}^\infty a_i, 0)$ assuming that $\sum_i a_i < \infty$.

Let $S \subset U$ be a subsemigroup of $U$.  We call $S$ a {\it
right limit semigroup} if for all $X^{(1)}, X^{(2)}, \ldots$ in $S$
such that $X = \prod_{i =1}^\infty X^{(i)}$ exists, we have $X \in
S$.  Similarly, we define a {\it left limit semigroup} by replacing
right infinite products with left infinite products.

Let us define the {\it right Chevalley group} to be the smallest
subset $\L_r \subset U_{\geq 0}$ satisfying
\begin{enumerate}
\item
every $e_i(a)$ for $a \geq 0$ lies in $\L_r$,
\item
if $X, Y \in \L_r$ then $XY \in \L_r$ (that is, $\L_r$ is a
semigroup), and
\item
$\L_r$ is a right limit semigroup.
\end{enumerate}
Note that $\L_r$ exists because we may define $\L_r$ to be the
intersection of all (non-smallest) subsets satisfying (1), (2) and
(3).  We say that $\L_r$ is the right limit semigroup generated by
$e_i(a)$. Similarly, we define $\L_l$, the {\it left Chevalley
group} to be the left limit semigroup generated by $e_i(a)$.

\begin{remark}
In \cite{LP} we shall show that elements of $\L_r$ (resp.~$\L_l$)
have ``canonical'' factorizations.
\end{remark}


\subsection{Factorization of doubly entire TNN matrices}
A TNN matrix $X \in U_{\geq 0}$ is {\it regular} if it is either (i)
the identity matrix, or (ii) doubly entire, infinitely supported and
satisfying $\epsilon_i(X) = \mu_i(X) = 0$ for every $i$.  For
example, the matrix in Example \ref{ex:1} is regular.

\begin{lem}\label{lem:entiresupport}
Suppose $X \in U_{\geq 0}$ is entire, and infinitely supported. Then
$X^{-c}$ is infinitely supported.
\end{lem}
\begin{proof}
Otherwise by Theorem \ref{thm:finitefactor}, $X^{-c}$ is a finite
product of possibly degenerate whirls.  If $X^{-c}$ is a product of
only Chevalley generators then $X$ will be finitely supported, so
the factorization of $X^{-c}$ must involve at least one
non-degenerate whirl.  But then by Lemma \ref{lem:infinitewhirldet},
$X$ would not be entire.
\end{proof}

\begin{lem}
Suppose $X$ is a doubly-entire infinitely supported TNN matrix.
Then $\epsilon_i(X) = 0$ for every $i$ if and only if $\mu_i(X^{-c})
= 0$ for every $i$.
\end{lem}
\begin{proof}
By Lemma \ref{lem:entiresupport} and Lemma \ref{lem:limits},
$X^{-c}$ is infinitely supported, so $\mu_i(X^{-c})$ is
well-defined.  By Lemma \ref{lem:ASWconv},
$(\epsilon_1(X),\ldots,\epsilon_n(X))$ records the parameters of the
biggest curl which can be factored out of $X$ on the left.
Similarly, $(\mu_1(X^{-c}),\ldots,\mu_n(X^{-c}))$ records the
parameters of the biggest curl which can be factored out of $X^{-c}$
on the right.  Because both $X$ and $X^{-c}$ is entire, such curls
are in fact products of Chevalley generators, and inverse of
Chevalley generators are Chevalley generators. So we have
$\epsilon_i > 0$ for some $i$, if and only if some Chevalley
generator can be factored out on the left of $X$, if and only if
some Chevalley generator can be factored out of $X^{-c}$ on the
right, if and only if $\mu_j(X^{-c}) > 0$ for some $j$.
\end{proof}

\begin{cor}
n A regular matrix $X \in U_{\geq 0}$ satisfies $\epsilon_i(X) =
\mu_i(X) = \epsilon_i(X^{-c}) = \mu_i(X^{-c}) = 0$ for every $i$.
\end{cor}

\begin{thm}\label{thm:regular}
Every doubly entire, infinitely supported, $X \in U_{\geq 0}$ can be
factorized as $X = A Y B$ where $A \in \L_r$, $B \in \L_l$ and $Y
\in U_{\geq 0}$ is regular.
\end{thm}

In \cite{LP}, we shall strengthen Theorem \ref{thm:regular} by
showing that the factorization is unique.

\begin{proof}
We use transfinite induction.  Every degenerate whirl or curl is a
product of Chevalley generators.  Pick such a factorization for each
degenerate whirl or curl, once and for all.

Now we define a $X_\alpha \in U_{\geq 0}$ for each ordinal $\alpha$.
We define $X_0 = X$.  We define $X_{\alpha+1}$ by factoring out a
Chevalley generator from $X_\alpha$ on both the left and the right
(if possible), always using the first Chevalley generator in the
chosen factorization of the curl specified by ASW factorization
(Lemma \ref{lem:ASW}).  If $X_\alpha$ is regular so that no
Chevalley generators can be factored out then $X_{\alpha+1} =
X_\alpha$. Finally, if $\alpha$ is a limit ordinal, then we set
$X_\alpha = \inf_{\beta < \alpha} X_\beta$, where the infimum is
taken entry-wise.

If $X_\alpha$ is never regular, then it is easy to see that $\alpha
\mapsto X_\alpha$ is injective ($X_\alpha$ is always decreasing).
This is impossible because $X_\alpha \in U$, and the cardinality of
$U$ is the same as that of the real numbers.  Thus $X_\alpha$ is
eventually regular, and this is the required matrix $Y$ of Theorem
\ref{thm:regular} (the matrices $A$ and $B$ are obtained by
remembering the Chevalley generators used during the transfinite
induction).
\end{proof}


\section{Commuting through infinite whirls and curls}
\label{sec:semigroup}
\subsection{(Limit) semigroups of infinite whirls and curls}

\begin{thm} \label{thm:semi}
Each of the sets $RC, LC, RW, LW$ of infinite products of non-degenerate
whirls and curls forms a semigroup.
\end{thm}
\begin{example} \label{ex:6}
Let $n = 2$.  Consider the infinite curl $X = \prod_{i \geq 0}
N(2^{-i}, 2^{-i-1})$. Then the entries of $X$ for $i<j$ are given by
$x_{i,j} = 2^{-\omega(i,j)} \prod_{k=1}^{j-i} \frac{2^k}{2^k-1}$,
where $\omega(i,j) = [(j-i)/2] + 1$ if $j$ is odd and $i$ is even,
and $\omega(i,j) =[(j-i)/2]$ otherwise.  A fragment of $X$ looks as
$$
\left(\begin{array}{ccccccc} \ddots & \vdots & \vdots & \vdots & \vdots &\vdots \\
\cdots& 1& 2 & \frac{4}{3} & \frac{32}{21} & \frac{256}{315}& \cdots \\
\cdots& 0& 1& 1 & \frac{4}{3} & \frac{16}{21} & \cdots \\
 \cdots& 0&0 & 1& 2& \frac{4}{3}&\cdots \\
 \cdots& 0&0&0&1& 1& \cdots \\
  \cdots& 0&0&0&0&1&  \cdots \\
 & \vdots & \vdots & \vdots & \vdots &\vdots &\ddots
\end{array} \right)
.
$$
One can check using the curl commutation relation that if $a/b =
c/d$ then $N(a,b) N(c,d) = N(c,d) N(a,b)$. Using that one verifies
that $X^2 = \prod_{i \geq 0} N(2^{-i}, 2^{-i-1})^2$.
\end{example}

We focus on the case of $RW$.  Theorem \ref{thm:semi} follows from
Lemma \ref{L:Z} below.

\begin{lem} \label{lem:pwhirl}
Let $X = M(a_1,\ldots,a_n)$ and $Y=M(b_1,\ldots,b_n)$ be two
non-degenerate whirls, such that $\max(a_i) < \epsilon < \min(b_i)$.
Define $Y' = M(b'_1,\cdots,b'_n)$, $X' = M(a'_1,\ldots,a'_n)$ to be the result of applying the whirl relation, so
that $XY = Y'X'$.  Then, for each $i$
$$
|b'_i - b_i| \leq \epsilon \frac{n\max(b_i)}{\min(b_i)}.
$$
\end{lem}
\begin{proof}
We have
\begin{align*}
b'_i = \frac{b_{i+1} \kappa_{i+1}({\bf a}, {\bf b})}{
\kappa_{i}({\bf a}, {\bf b})} > \frac{\prod_{i=1}^n
b_i}{\kappa_{i}({\bf a}, {\bf b})} &= \frac{b_i}{1 +
\sum_{j=i}^{i+n-2} \prod_{k=j+1}^{i+n-1}(a_k/b_k)} \\
&> b_i(1 - n(a_{i+1}/b_{i+1}))
> b_i - \epsilon \frac{n\max_i(b_i)}{\min_i(b_i)}.
\end{align*}
Similarly, $b_i > b'_i - \epsilon \frac{n\max_i(b_i)}{\min_i(b_i)}$.
\end{proof}

\begin{lem}\label{L:multiY}
Let $Y_1, Y_2, \ldots, Y_r$ be non-degenerate whirls and $\delta >
0$. Let $X_1, X_2, \ldots, X_m$ be a finite sequence of whirls, and
let $Y'_j$ be obtained by successively commuting $Y_j$ through the
$X_i$:
$$
X_1 X_2 \cdots X_m Y_1 Y_2 \cdots Y_r = Y'_1 Y'_2 \cdots Y'_r X'_1
X'_2 \cdots X'_m.
$$
Then there is a constant $C$, depending only on $Y_1, \ldots, Y_r$,
such that if the total sum of parameters in all the $X_i$ is less
than $C$, then for each $i$, the parameters in $Y_i$ differ from
those in $Y'_i$ by at most $\delta$.
\end{lem}

\begin{proof}
Lemma \ref{lem:pwhirl} allows us to pull the $Y$-s through the
$X$-s, one after another, guaranteeing that the parameters in the
$Y_i$-s do not change too much. While doing that we need to know
that the parameters inside $X$-s remain small so that we can
repeatedly apply Lemma \ref{lem:pwhirl}. This however follows from
the fact that parameters in $Y$-s do not change much, while the
total sum of parameters in $X$-s and $Y$-s remains constant.
\end{proof}

Let $X = \prod_{i=1}^\infty X_i$ and $Y =\prod_{i=1}^\infty Y_i$ be
two infinite products of whirls.  We assume the products are written
in the canonical ASW order, that is we have $r(X^{-c}_1) \leq
r(X^{-c}_2) \leq \cdots$ and similarly for $Y$.  We will call
$r(X^{-c})$ the inverse radius of convergence of $X$.  For each $m
\in \{1,2,\ldots\}$, let $X_1, \cdots,X_{s_m}$ and $Y_1, \cdots,
Y_{t_m}$ be the factors with inverse radius of convergence less than
$m$. We may rewrite using the ASW factorization
$$
X_1 X_2 \cdots X_{s_m} Y_1 Y_2 \cdots Y_{t_m} = Z^{(m)}_1 Z^{(m)}_2
\cdots Z^{(m)}_{s_m+t_m}.
$$
Each of the matrices $Z^{(1)}_1, Z^{(2)}_1, \cdots$ is a whirl
depending on $n$ real parameters.  These parameters are bounded
above by the sum of the parameters in $X$ and $Y$, so a subsequence
of $\{Z^{(m_i)}_1\}$ of converges to some whirl $Z_1$, which must be
non-degenerate.  Now find a subsequence of the matrices
$\{Z^{(m_i)}_2\}$ which converge to a whirl $Z_2$, and repeat to
define $Z_1, Z_2 , \cdots$.

\begin{lem}\label{L:Z}
The infinite product $Z = \prod_{i=1}^\infty Z_i$ converges, and $Z
= XY$.
\end{lem}
\begin{proof}
We first show that $Z_1 Z_2 \cdots Z_k \leq XY$ entrywise.  This can
be done by finding a sufficiently large $m$ so that $Z_i^{(m)}$ is arbitrarily close to $Z_i$, for each $i =
1,2,\ldots,k$, in any desired set of entries.  Then we have
$$
\prod_{i=1}^k Z_i \sim \prod_{i=1}^k Z^{(m)}_i \leq
\prod_{i=1}^{s_m+t_m} Z^{(m)}_i = X_1 X_2 \cdots X_{s_m} Y_1 Y_2
\cdots Y_{t_m} \leq XY
$$
where the inequalities are entrywise.

Conversely, we show that for each $j, k$, we have $X_1 X_2 \cdots
X_{j} Y_1 Y_2 \cdots Y_{k} \leq Z$.  Pick $r > j$ so large that the
sum of all the parameters in $X_r, X_{r+1}, \ldots$ is less than the
constant $C$ of Lemma \ref{L:multiY}, for some small $\delta > 0$.
Now pick $m$ sufficiently large so that the inverse radius of
convergence of $X_1, \ldots, X_r, Y_1, \ldots, Y_k$ are all less
than $m$; in other words, $s_m \geq r$ and $t_m \geq k$.  Pick $m'$
sufficiently large so that $Z_i$ and $Z_i^{(m')}$ are arbitrarily
close for all $i \leq (s_m + t_m)$. Then
\begin{align*}
\prod_{i=1}^{s_m+t_m} Z_i &\sim  \prod_{i=1}^{s_m+t_m} Z^{(m')}_i \\
&= X_1 X_2 \cdots X_{s_m} Y'_1 Y'_2 \cdots Y'_{t_m} \\
& \geq X_1 X_2 \cdots X_{s_m} Y'_1 Y'_2 \cdots Y'_{k} \\
& \sim X_1 X_2 \cdots X_{s_m} Y_1 Y_2 \cdots Y_{k} \\
 &\geq X_1 X_2 \cdots X_{j} Y_1 Y_2 \cdots Y_{k}
\end{align*}
where $Y'_1, Y'_2, \ldots,Y'_{t_m}$ is obtained by commuting $Y_1,
Y_2, \ldots, Y_{t_m}$ past $X_{s_m +1}, \ldots, X_{s_{m'}}$.  Again
the approximations denoted by $\sim$ mean that a finite set of
entries is arbitrarily close.
\end{proof}

Essentially the same proof establishes a stronger statement. Recall
the definition of right and left limit semigroups from subsection
\ref{ssec:chev}.

\begin{thm} \label{thm:lsemi}
The semigroups $RC$ and $RW$ (resp.~$LC$ and $LW$) of infinite
products of non-degenerate whirls and curls are right (resp.~left)
limit semigroups.
\end{thm}

\begin{proof}
We prove the statement for $RW$.  Assume we have an infinite product
of infinite whirls: $V U \cdots$, each of which has been written in
the canonical ASW order.  For each $m \in \{1,2,\ldots\}$ denote by
$v(m)$ the index such that $V_1, \ldots, V_{v(m)}$ are exactly the
whirls in $V = \prod_{i=1}^\infty V_i$ with inverse radius of
convergence smaller than $m$.  Similarly define $u(m)$ for $U$, and
so on.  Note that for each $m$ only finitely many of the factors $V,
U, \ldots$ contain a whirl with inverse radius smaller than $m$. For
each $m$, we define non-degenerate whirls $Z^{(m)}_i$ by the
following equality:
$$V_1 \cdots V_{v(m)} U_1 \cdots U_{u(m)} \cdots W_1 \cdots
W_{w(m)} X_1 \cdots X_{x(m)} Y_1 \cdots Y_{y(m)} = Z_1^{(m)} \cdots
Z_{v(m)+\cdots + y(m)}^{(m)}$$ where the $Z^{(m)}_i$ are in the
canonical ASW order.

As before the proof of Lemma \ref{L:Z}, choose subsequences of $m$-s
to define $Z_1, Z_2, \ldots $.  We now claim that $V U \cdots = Z_1
Z_2 \cdots$.  For the inequality $Z_1 \cdots Z_q \leq V U \cdots$,
the proof is the same as in Lemma \ref{L:Z}.  For the other
direction let us assume we are given a product
$$V_1 \cdots V_v U_1 \cdots U_u \cdots W_1 \cdots W_w X_1
\cdots X_x Y_1 \cdots Y_y$$ of initial parts of certain finite
number of initial factors.  We now repeatedly apply Lemma
\ref{L:multiY}, in a similar manner to the proof of Lemma \ref{L:Z}.
Namely, choose $m_1$ so that
$$Y_1 \cdots Y_y \sim Y_1' \cdots Y_y',$$ where $Y_1' \cdots Y_y'$
is obtained by commuting $Y_1 \cdots Y_y$ through $X_{x(m_1)+1}
\cdots X_N$ for some $N$.  We may assume that $x(m_1) > x$.  By
Lemma \ref{L:multiY} we may assume $m_1$ is chosen so that the
approximation holds for any $N$.

Similarly choose $m_2 > m_1$ so that $$X_1 \cdots X_{x(m_1)} Y_1'
\ldots Y_{y}' \sim X_1' \cdots X_{x(m_1)}' Y_1'' \ldots Y_{y}''.$$
Here
$X_1' \cdots X_{x(m_1)}' Y_1'' \ldots Y_{y}''$ is obtained by
pulling $X_1 \cdots X_{x(m_1)} Y_1' \ldots Y_{y}'$ through the
product $W_{w(m_2)+1} \cdots W_{N}$.  Again we assume that $w(m_2) >
w$.  On the next step we find $m_3 > m_2$ that would allow to pull
$$W_1 \cdots W_{w(m_2)} X_1' \cdots X_{x(m_1)}' Y_1'' \ldots
Y_{y}'' $$ through the next factor, and so on.  Finally let $m =
\max(m_i)$ be the parameter in the last move and find $m'$ so that
$\prod_{i=1}^{v(m)+ \cdots + y(m)} Z^{(m')}_i$ is arbitrarily close
to $\prod_{i=1}^{v(m)+ \cdots + y(m)} Z_i$.  Now we calculate
\begin{align*}
\prod_{i=1}^{v(m)+ \cdots + y(m)} Z_i &\sim \prod_{i=1}^{v(m)+
\cdots + y(m)} Z^{(m')}_i \\
& =  V_1 \cdots V_{v(m)} \cdots W^*_1 \cdots
W^*_{w(m_2)} X^*_1 \cdots X^*_{x(m_1)} Y_1^* \cdots Y_y^* A\\
&\geq V_1 \cdots V_{v(m)} \cdots W^*_1 \cdots W^*_{w(m_2)} X^*_1
\cdots X^*_{x(m_1)} Y_1^* \cdots Y_y^* \\
&\sim V_1 \cdots V_{v(m)} \cdots W_1 \cdots W_{w(m_2)} X_1 \cdots
X_{x(m_1)} Y_1 \cdots Y_y \\
&\geq V_1 \cdots V_v \cdots W_1 \cdots W_w X_1 \cdots X_x Y_1 \cdots
Y_y.
\end{align*}
We explain the equality on the second line.  Here $W^*_i, X^*_i,
Y^*_i$ denote what we get when we commute $Y_1 \cdots Y_y$ past
$X_{x(m_1)+1} \cdots X_{x(m')}$, and then commute $X_1 \cdots
X_{x(m_1)}Y_1 \cdots Y_y$ past $W_{w(m_2)+1} \cdots W_{w(m')}$, and
so on.  Applying to $\prod_{i=1}^{v(m')} V_i
\cdots \prod_{i=1}^{w(m')}W_i\prod_{i=1}^{x(m')}X_i
\prod_{i=1}^{y(m')} Y_i$ all these commutations we obtain $V_1 \cdots V_{v(m)} \cdots X^*_1
\cdots X^*_{x(m_1)} Y_1^* \cdots Y_y^* B$ where $B$ consists of the
whirls obtained from
$$V_{v(m)+1}, \ldots, V_{v(m')},\ldots,
W_{w(m_2)+1},\ldots,W_{w(m')}, X_{x(m_1)+1},\ldots, X_{x(m')},
Y_{y+1},\ldots,Y_{y(m')}$$ via commutation.  The matrix $A$ is what
we get when we in addition commute all the whirls in $B$ with
inverse radius of convergence greater than $m$ to the right and
remove them.
\end{proof}

\subsection{Chevalley generators out of whirls}
We have shown that $RC, LC, RW,$ and $LW$ are semigroups.  We now
describe what happens when they are multiplied by Chevalley
generators from a particular side.  We only state our results for
right-infinite whirls and curls.

\begin{thm} \label{T:Chevcurl}
Suppose $e_i(a)$ is a Chevalley generator and $X \in RW$ (resp.~$X
\in RC$).  Then $e_i(a)X \in RW$ (resp.~$e_i(a)X \in RC$).
\end{thm}

\begin{example}
If $X$ is the right-infinite curl in Example \ref{ex:6} then $$e_1(1) X = N(2, \frac{1}{4}) N(\frac{1}{4},\frac{1}{2}) N(\frac{1}{2},\frac{1}{16}) N(\frac{1}{16},\frac{1}{8}) \dotsc = \prod_{i \geq 0} N(2^{1-2i}, 2^{-2-2i}) N(2^{-2-2i}, 2^{-1-2i}).$$
\end{example}

Theorem \ref{T:Chevcurl} follows from the following more precise
Lemma.

\begin{lemma} \label{lem:chthrw}
Let $e_i(a)$ be a Chevalley generator and $X = \prod_{k=1}^{\infty}
M(b^{(k)}_1, \ldots, b^{(k)}_n)$ be a right-infinite product of
non-degenerate whirls.  Using the whirl commutation relation of
Theorem \ref{T:commute}, we define $\c^{(k)} = (c^{(k)}_1, \ldots,
c^{(k)}_n)$ and $a^{(j)}$ by \begin{equation}\label{E:abc} e_i(a)
\prod_{k=1}^{\infty} M(\b^{(k)}) = \prod_{k=1}^j M(\c^{(k)})
e_{i-j}(a^{(j)}) \prod_{k=j+1}^{\infty} M(\b^{(k)}).\end{equation}
Then
\begin{enumerate}
\item $\lim_{j \to \infty} a^{(j)} = 0$
\item The product $\prod_{k=1}^{\infty} M(\c^{(k)})$ is well-defined and equals $X$.
\end{enumerate}
The analogous statement holds for curls.
\end{lemma}

We may think of Lemma \ref{lem:chthrw} as saying that infinite
products $\prod_{i=1}^{\infty} M^{(j)}_i$ of whirls (or curls)
``absorb'' Chevalley generators (if multiplied on the correct side).

\begin{proof}
In the setting of Lemma \ref{lem:chevcom} one has
$a'_{i+1}=\frac{b_{i+1}a_i}{a_i+b_i} < b_{i+1}$. In order for the
product $X$ to be well-defined it must be the case that $\lim_{j \to
\infty} b^{(j)}_{i-j}=0$, and so $\lim_{j \to \infty} a^{(j)} = 0$,
proving the first statement.

For the second part, consider a fixed entry $x_{s,t}$.  Suppose that
the sequence $X_j = e_i(a) \prod_{k=1}^{j} M(b^{(k)}_1, \ldots,
b^{(k)}_n)$ of matrices has entries $m_j$ in location $(s,t)$. Then
$\lim_{j \to \infty} m_j = x_{s,t}$.  Similarly define $m'_j$ as the
corresponding entry of $X'_j = \prod_{k=1}^{j} M(c^{(k)}_1, \ldots,
c^{(k)}_n).$  Clearly $\lim_{j \to \infty} m'_j$ exists and is less
than $x_{s,t}$.  We must show that the limit equals $x_{s,t}$.

For a given $\delta > 0$ one can choose $j$ large enough so that
$a^{(j)} x_{s,t-1} < \delta/2$ and $x_{s,t} - m_j < \delta/2$.  The
equality \eqref{E:abc} shows that $m_j - m'_j \leq a^{(j)}
x_{s,t-1}$, so we deduce that $x_{s,t} - m'_j < \delta$.  Thus
$\lim_{j \to \infty} m'_j = x_{s,t}$.

The proof for curls is verbatim, using the inequality
$a'_{i-1}=\frac{b_{i-1}a_i}{a_i+b_i} < b_{i-1}$ from Lemma
\ref{lem:chevcom2}.
\end{proof}

\subsection{Not all Chevalley generators at once}
The $\epsilon$-sequence of a TNN matrix $X$ give a bound on what
Chevalley generators can be factored out from $X$ on the left so
that the result remains TNN. In particular, by Lemma
\ref{lem:ASWconv}, $e_i(a)$ cannot be factored out if $a >
\epsilon_i$.  This bound is far from sharp: for example no Chevalley
generator can be factored out from a non-degenerate curl, but every
$\epsilon_i$ of a curl is strictly positive.

\begin{proposition}
Let $X \in U_{\geq 0}$.  There is an $i \in \Z/n\Z$ such that if $X
= e_i(a) X'$ for $a \geq 0$ and $X' \in U_{\geq 0}$ then $a = 0$.
\end{proposition}

\begin{proof}
Assume the statement is false and that for each $j$ we have $X =
e_j(a_j) X_j$ for some TNN $X_j$-s and $a_j > 0$.  By Theorem
\ref{thm:canon} one can write $X_j=\prod_{i=1}^{\infty} N^{(j)}_i
E^{(j)}$ where $E^{(j)}$ is entire.  There are two cases to
consider.

Case (1).  One of the products $\prod_{i=1}^{\infty} N^{(j)}_i$ has
only finitely many non-trivial terms.  Then one can commute
$e_j(a_j)$ through this product to obtain another finite product of
curls times $e_{j'}(a_{j'}) E^{(j)}$, which is entire.  Since the
decomposition of Theorem \ref{thm:canon} is unique, this means by
Lemma \ref{lem:chthrw} that the products $\prod_{i=1}^{\infty}
N^{(j)}_i$ are finite for each $j \in \Z/n\Z$ and that the
corresponding expressions $e_{j'}(a_{j'}) E^{(j)}$ are all equal to
some entire matrix $E$ (what we get from $X$ by removing the curl
component of $X$).  As $j$ varies over $\Z/n\Z$, so does $j'$.
Furthermore, each $a_{j'} > 0$.  This is impossible, because $E$,
being entire, has one of the $\epsilon$-s equal to $0$, and the
corresponding Chevalley generator cannot be factored out with any
positive constant.

Case (2).  All the products $\prod_{i=1}^{\infty} N^{(j)}_i$ are
infinite. Let $X=\prod_{i=1}^{\infty} N_i E$ factorize $X$ into its
curl component and an entire matrix.  By Lemma \ref{lem:chthrw} and
by the uniqueness in Theorem \ref{thm:canon} we have
$\prod_{i=1}^{\infty} N_i = e_j(a_j) \prod_{i=1}^{\infty} N^{(j)}_i$
for every $j$. Without loss of generality we can assume that each
such product of curls is an ASW factorization.

Let us consider what happens to the $a_j$ when we commute $e_j(a_j)$
past $N^{(j)}_1 = N(\b^{(j)})$.  We know that $e_j(a_j)N(\b^{(j)}) =
N(\b') e_{j'}(a'_j)$, where $N_1 = N(\b')$ does not depend on $j$.
We calculate using Lemma \ref{lem:chevcom2} that $a'_{j-1}/a_j =
b'_{j-1}/(b'_j-a_j) > b'_{j-1}/b'_j$. Note that there are no
references to $\b^{(j)}$ in these inequalities.

Now we observe that
$$
\frac{\prod_{j \in \Z/n\Z}a'_j}{\prod_{j \in \Z/n\Z} a_i} = \prod_{j
\in \Z/n\Z} \frac{a'_{j-1}}{a_j} > \prod_{j \in \Z/n\Z}
\frac{b'_{j-1}}{b'_j} =1.
$$
Thus the total product of parameters in the $e_j(a_j)$ increases
after commuting past $N^{(j)}_1$.  The same argument shows that the
product of parameters will continue to increase as we commute past
$N^{(j)}_2, N^{(j)}_3, \ldots$.  This contradicts Lemma
\ref{lem:chthrw}, which says that all $n$ Chevalley parameters have
zero limit.
\end{proof}

\subsection{Pure whirls and curls}
Let us call $X \in RW \cap LW$ a {\it pure whirl}, and write $PW =
RW \cap LW$.  Similarly we define the set $PC$ of {\it pure curls}.
By Theorem \ref{thm:semi}, we have
\begin{example} \label{ex:7}
The right-infinite curl $X$ from Example \ref{ex:6} is pure. Indeed,
one can see that $X$ has southwest-northeast axes of symmetry, and
thus its factors could be multiplied in the reverse direction: $X =
\prod_{i = -\infty}^{0} N(2^{-i}, 2^{-i-1})$.  One can also derive
this from the fact that the curl factors in $X$ commute.
\end{example}

\begin{thm}
The sets $PW$ and $PC$ of pure whirls and curls are semigroups.
\end{thm}

Certain properties of pure whirls and curls are immediately clear,
for example it follows from Lemma \ref{lem:whirlmu} that elements of
$PW$ have all $\epsilon_i$-s and $\mu_i$-s equal to $0$.  We state
the following result only for infinite curls.  The result for whirls
is obtained by applying $^{-c}$.

\begin{theorem} Each $X \in RC$ can be uniquely factored as $X
= E X'$, where $E$ is doubly-entire and $X' \in PC$.  Similarly,
each $X \in LC$ can be uniquely factored as $X = X'E$, where $E$ is
doubly-entire and $X' \in PC$.
\end{theorem}

\begin{proof}
We consider the case of $LC$, the case of $RC$ being identical.
Apply Theorem \ref{thm:canon} to obtain $X = X' E$ where $E$ is
entire and $X' \in RC$.  The matrix $E$ must be doubly-entire, for
otherwise a non-degenerate whirl can be factored out of $X$ on the
right. But this would mean that a non-degenerate curl can be
factored out of $X^{-c}$ on the left. This is impossible by Lemma
\ref{lem:whirlmu}, since $X^{-c}$ is an infinite product of whirls.

The factorization $X = X'E$ is unique, so it remains to show that
$X' \in PC$.  Apply (left-right swapped) Theorem \ref{thm:canon} to
$X'$ to rewrite it as $X' = F X''$, where $X'' \in LC$ and $F$ is
entire.  Finally, rewrite $X'' E $ as $G X'''$ where $X''' \in LC$
and $G$ is entire.  In the end we get $X = FGX'''$.  By Theorem
\ref{thm:canon} and the assumption that $X \in LC$, the entire
matrix $GF$ must be trivial, and thus $F$ is trivial.  This means
exactly $X' \in LC$.
\end{proof}

\section{Minor ratio limits}\label{sec:fact}
\subsection{Ratio limit interpretation and factorization problem}
Let $X \in RC$ and let $X = \prod_{i=1}^{\infty} N(a_1^{(i)},
\ldots,a_n^{(i)})$ be the ASW factorization of $X$.  Let $k \geq 1$
be an integer.  Let $I = \{i_1 < i_2 < \ldots < i_k\}$ be a
collection of positive integers such that $i_t \leq i+t$ for an
integer $i$, and let $I^k_i = \{i+1,i+2,i+3, \ldots, i+k\}$.  Also
let $J^k_h = \{h+1,h+2, \ldots, h+k\}$.  By Theorem \ref{T:schur},
the minor $\Delta_{I, J^k_j}(X)$ is equal to
$s_{\lambda_{I,h}^k}(\a)$, where $\lambda_{I,h}^k = \lambda(I,
J^k_h)$ is a skew shape the right-hand side of which is vertical. We
also define $\mu_{i,h}^k = \lambda(I^k_i,J^k_h)$, which is a
rectangular skew shape of height $k$ and width $h - i$.

It is clear that $\mu_{i,j}^k \subset \lambda_{I,h}^k$.  We let $\nu
= (i+k,\ldots,i+k)/(i_k, \ldots, i_1+k-1)$ be the difference.  We
define
$$
s_\nu(\a^{(1)},\a^{(2)},\ldots,\a^{(k)}) = \sum_T a^T
$$
to be the weight generating function of tableaux with shape $\nu$,
and filled with numbers from $[1,k]$.  We can also obtain
$s_\nu(\a^{(1)},\a^{(2)},\ldots,\a^{(k)})$ from $s_\nu(\a)$ by
setting $a^{(i)}_j = 0$ for $i > k$.  For example, for $I = (i, i+2,
\ldots, i+k)$ we have $s_\nu(\a^{(1)},\a^{(2)},\ldots,\a^{(k)}) =
\sum_{j=1}^k a_i^{(j)}$.

\begin{theorem} \label{thm:mrl}
We have
$$s_\nu(\a^{(1)},\a^{(2)},\ldots,\a^{(k)}) = \lim_{h \to \infty}
\frac{s_{\lambda_{I,h}^k}(\a)}{s_{\mu_{i,h}^k}(\a)} = \lim_{h \to
\infty} \frac{\Delta_{I, J_h^k}(X)}{\Delta_{I^k_i, J_h^k}(X)}.$$
\end{theorem}

\begin{proof}
The general plan is similar to the proof of Lemma \ref{L:aep}, but
the details are significantly more complicated.

We can write $$\frac{s_{\lambda_{I,h}^k}(\a)}{s_{\mu_{i,h}^k}(\a)} =
s_\nu(\a^{(1)},\a^{(2)},\ldots,\a^{(k)}) +
\frac{\wt(S_h^{k})}{\Delta_{I^k_i, J_h^k}(X)},$$ where $S_h^{k}$ is
the set of all semistandard fillings of $\lambda_{I,h}^{k}$ such
that not all numbers filling the left $\nu$ part of the shape are in
the range from $1$ to $k$. It remains to show that $\lim_{h \to
\infty} \frac{\wt(S_h^{k})}{\Delta_{I^k_i, J_h^k}(X)} = 0$.  Denote
by $T_h^{k}$ the set of all semistandard fillings of $\mu_{i,h}^{k}$
with entries in the bottom row not smaller than $k+1$.  Since
$\wt(S_h^{k}) < s_{\nu}(\a) \wt(T_h^{k})$, it suffices to show that
$$\lim_{h \to \infty} \frac{\wt(T_h^{k})}{\Delta_{I^k_i, J_h^k}(X)} =
0.$$

We shall prove by induction a stronger statement.  Namely, let us
take a vector $b = (b_1, \ldots, b_{h+k})$ of positive integers we
call {\it {bounds}}.  We also allow some of the bounds to be
infinite.  Denote by $T_h^{k,b}$ the set of all semistandard
tableaux of shape $\mu_{i,h}^{k}$ with entries in the bottom row not
smaller than $k+1$ and smaller than the corresponding entries of $b$
(that is, an entry in the $r$-th column has to be less than or equal
to $b_r$).  One can think of $b$ as a hidden $(k+1)$-st row of the
tableau.  Similarly denote by $U_h^{k,b}$ the set of all
semistandard fillings of $\mu_{i,h}^{k}$ with the first entry in the
bottom row equal $k$ and all entries in the bottom row less than the
corresponding entry of $b$.  Let $V_h^{k,b} = U_h^{k,b} \cup
T_h^{k,b}$ be the set of all semistandard fillings of
$\mu_{i,h}^{k}$ with entries in the bottom row less than the
corresponding entry of $b$.  We claim that for a fixed $\varepsilon
> 0$ there is $N$ such that for $h \geq N$ we have
$\wt(T_h^{k,b})/\wt(U_h^{k,b}) < \varepsilon$, or equivalently
$\wt(V_h^{k,b})/\wt(U_h^{k,b}) < 1 + \varepsilon$ for any $b$ such
that $V_h^{k,b}$ (and thus $U_h^{k,b}$) is non-empty.

We proceed by induction on $k$.  Checking the base case $k = 1$ is
essentially the same as checking the inductive step, so assume now
that the claim has been proved for the values up to $k$, and prove
it for the $(k+1)$-row case.  By the induction assumption, for any
$\varepsilon$ there exists an $N$ such that for $h \geq N$ and any
bound $b'$, the fillings of the first $k$ rows with the first column
filled with the numbers $1, \ldots, k$ constitute at least $1/(1 +
\varepsilon)$ part of weight of all possible fillings.  Iterating,
we can claim that for any $m$ and $\varepsilon$ there exists an $N$
such that for $h \geq N$ and any bound $b'$ the fillings of the
first $k$ rows with the first $m$ columns filled minimally
constitute at least $1/(1 + \varepsilon)$ portion of weight of all
possible fillings.  Thinking of the bounds $b'$ as a $(k+1)$-st row,
we now sum over all $b'$ which are compatible with given bound $b$,
and conclude that for any $m$ and $\epsilon$ there is an $N$ such
that for $h \geq N$
$$\frac{\wt(V_h^{k+1,b})}{\wt(W_{m,h}^{k+1,b})} < 1 + \varepsilon.$$
Here $\wt(W_{m,h}^{k+1,b})$ denotes all fillings of
$\mu_{i,h}^{k+1}$ compatible with $b$ such that the rectangle formed
by first $k$ rows and first $m$ columns is filled minimally, that
is, with the numbers $1, 2, \ldots, k$.  Now among let
$T_{m,h}^{k+1,b} \subset W_{m,h}^{k+1,b}$ be the subset of tableaux
with the lowest row filled with numbers greater than $k+1$, and the
$U_{m,h}^{k+1,b}  \subset W_{m,h}^{k+1,b}$ be the subset of tableaux
with the lower left corner filled with $k+1$.  Note that
$\wt(U_{m,h}^{k+1,b}) < \wt(U_{h}^{k+1,b})$ since dropping the
minimality condition on the second to $m$-th rows can only increase
the sum.

Pick $R$ so that $$\sum_{\ell \geq R} \sum_{j=1}^n a_j^{(\ell)} <
\min(a_1^{(k+1)} a_2^{(k+1)},\ldots, a_n^{(k+1)}).$$ This can be
done since the sum of all $a_j^{(\ell)}$ is finite.  Let
$Q_{m,h}^{k+1,b} \subset T_{m,h}^{k+1,b}$ be the subset of tableaux
with only the labels $1,2,\ldots,R$ in the first $m$ columns.  We
define a map $T_{m,h}^{k+1,b} \to Q_{m,h}^{k+1,b}$ by changing every
entry in the last row and first $m$ columns which is greater than
$R$, to $R$.  As we did in Lemma \ref{L:aep}, we give tableaux in
$Q_{m,h}^{k+1,b}$ a modified weight, denoted $\wt'$: the entries in
a cell with residue $j$, in the last row and first $m$ columns,
labeled $R$, have weight equal to $a^{(k+1)}_j$.  All the other
entries have the usual weight.  By our choice of $R$, we have
$\wt(T_{m,h}^{k+1,b}) < \wt'(Q_{m,h}^{k+1,b})$.

For any $T \in Q_{m,h}^{k+1,b}$, there is some $r \in [k+1,R]$ such
that there are at least $m/R$ cells filled with $r$ in the last row.
If there are several options for $r$ choose the smallest one.  Let
us change the last row by removing the first $n, 2n, \ldots,$ of the
$r$'s, changing them to $(k+1)$'s placed in the front of the row. As
a result we get a filling that agrees with the bound $b$ since the
entry of each cell did not increase.  This produces $m/Rn$ distinct
tableaux in $U_{h}^{k+1,b}$.  The weight of the resulting tableau is
at least as large as the modified weight of the original one: if $r
< R$ this follows from the fact that in an ASW factorization the
products of parameters in successive curls do not increase.  If $r =
R$ this follows by definition of the modified weight.

Thus we obtain a multi-valued map from $Q_{m,h}^{k+1,b}$ to
$U_{m,h}^{k+1,b}$ such that each element of $Q_{m,h}^{k+1,b}$ maps
into $m/Rn$ elements of $U_{m,h}^{k+1,b}$, while each element of
$U_{m,h}^{k+1,b}$ is the image of less than $R$ elements of
$T_{m,h}^{k+1,b}$.  Thus we have
$$\wt(T_{m,h}^{k+1,b}) < \frac{m}{Rn} \wt'(Q_{m,h}^{k+1,b}) < R \; \wt(U_{m,h}^{k+1,b}),$$ which implies
$$\wt(W_{m,h}^{k+1,b}) < \left(1+\frac{R^2n}{m} \right) \wt(U_{m,h}^{k+1,b}).$$
Now we can combine several claims to get
$$\wt(V_h^{k+1,b}) < (1 +
\varepsilon)\wt(W_{m,h}^{k+1,b}) < (1 + \varepsilon)
(1+\frac{R^2n}{m}) \wt(U_{m,h}^{k+1,b}) \leq (1 + \varepsilon)
(1+\frac{R^2n}{m}) \wt(U_{h}^{k+1,b}).$$ Clearly for any $\delta >
0$ one can choose $\epsilon > 0$ and large enough $m$ so that $(1 +
\epsilon) (1+\frac{R^2n}{m}) < 1 + \delta$, which finishes the
proof.
\end{proof}

\begin{remark}
In \cite{LP}, we shall give a different interpretation of limit
ratio minors for arbitrary TNN matrices, not just for infinite
products of curls.
\end{remark}

\begin{example}
The definition of $\epsilon_i$ as the limit $\lim_{j \to \infty} \frac{x_{i,j}}{x_{i+1,j}}$ is an instance of Theorem \ref{thm:mrl} with $k=1$ and $I = \{i\}$.
\end{example}

\begin{example}
Take the matrix from Example \ref{ex:4}. Take $i=2$, $k= 2$, and $I
= (1,4)$. Then $$\lim_{h \to \infty} \frac{\Delta_{I,
J_h^k}(X)}{\Delta_{I^k_i, J_h^k}(X)} = \lim_{g \to \infty}
\frac{\det \left(\begin{array}{cc}
2^{g+2}-3 & 2^{g+2}-2 \\
3\cdot 2^{g-1} -3 & 3 \cdot 2^{g-1}-2
\end{array} \right)
}{\det
\left(\begin{array}{cc}
2^{g+1}-3 & 2^{g+1}-2 \\
3\cdot 2^{g-1} -3 & 3 \cdot 2^{g-1}-2
\end{array} \right)
} =5.$$ And indeed, this is the value of $s_{(4,4)/(4,2)} =
h_2^{(1)}$ evaluated at the first (and in this case - the only) two
curls of the ASW factorization: $\frac{4}{3}\cdot\frac{3}{2} +
\frac{4}{3}\cdot\frac{3}{2} + \frac{2}{3}\cdot\frac{3}{2} = 5$.
\end{example}

The proof of Theorem \ref{thm:mrl} clearly works for $X$ a finite
product of curls as long as $k$ is not larger than the number of
curls in the product. The following immediate corollary allows to
express all the parameters involved in the ASW factorization of an
infinite curl directly through the minor ratio limits.

\begin{corollary}\label{C:mrl}
We have $$a_i^{(k)} = \lim_{h \to \infty} \frac{\Delta_{I_{i-1}^k,
J_h^k}(X)}{\Delta_{I^k_i, J_h^k}(X)}\Big/\lim_{h \to \infty}
\frac{\Delta_{I_i^{k-1}, J_h^{k-1}}(X)}{\Delta_{I^{k-1}_{i+1},
J_h^{k-1}}(X)}.$$
\end{corollary}

\begin{proof}
By Theorem \ref{thm:mrl} the numerator is equal to
$\prod_{j=i+k-1}^{i} a^{(i+k-j)}_j$ and the denominator is equal to
$\prod_{j=i+k-1}^{i+1} a^{(i+k-j)}_j$, from which the statement
follows.
\end{proof}

It appears that even in the case $n=1$ the result of Theorem
\ref{thm:mrl} is new, we state it separately as follows. Let ${\bf
{a}} = a_1, a_2, \ldots$ be a sequence of parameters such that
$\sum_i a_i < \infty$ and let $s_{\lambda}$ denote the usual Schur
function. Let $\nu = (i+k,\ldots,i+k)/(i_k, \ldots, i_1+k-1)$ and
adopt other notation as above.
\begin{corollary}
The limit of ratios of Schur functions $\lim_{h \to \infty}
s_{\lambda_{I,h}^k}(\a)/s_{\mu_{i,h}^k}(\a)$ is equal to the Schur
polynomial $s_\nu(a_{i_1},a_{i_2},\ldots,a_{i_k})$ evaluated at the
$k$ largest parameters among the $a_i$-s.
\end{corollary}

\subsection{Invariance}

In \cite{RS} to any non-crossing matching $\tau$ on $2n$ vertices
and to any permutation $w \in S_n$ a number  $f_{\tau}(w)$ was
associated using the Temperley-Lieb algebra. Let $Y = (y_{st})$ be
an $n \times n$ matrix variables.  One can then construct a family
of polynomials
$$
\Imm_{\tau}^{\mathrm{TL}}(Y):=\sum_{w \in S_n} f_{\tau}(w)\,y_{1,w(1)} \cdots y_{n,w(n)}
$$
called {\it {Temperley-Lieb immanants}}.  Let us consider $2n$
points $\{1,2,\ldots,2n\}$ arranged in two columns, with the numbers
$\{1,2,\ldots,n\}$ arranged from top to bottom in the left column,
and the numbers $\{n+1,n+2,\ldots,2n\}$ arranged from top to bottom
in the right column.  A (complete) matching of $[2n]$ is called
non-crossing if it can be drawn without intersecting edges, where
edges are not allowed to leave the rectangle bounded by
$1,n,n+1,2n$.  For a subset $S\subset [2n]$, let us say that a
non-crossing (complete) matching is {\it $S$-compatible\/} if each
strand of the matching has one endpoint in $S$ and the other
endpoint in its complement $[2n]\setminus S$. Coloring vertices in
$S$ black and the remaining vertices white, a non-crossing matching
is $S$-compatible if and only if each edge in it has endpoints of
different color. Let $\Theta(S)$ denote the set of all
$S$-compatible non-crossing matchings. An example for $n=5$, $S =
\{3,6,7,8,10\}$ is shown in the figure below.
\begin{figure}[h!]
    \begin{center}
    \input{loo5.pstex_t}
    \end{center}
\end{figure}
For a subset $I \subset [n]$ let $\bar I := [n]\setminus I$ and let
$I^\wedge := \{2n+1-i\mid i\in I\}$. The following results were
obtained in \cite{RS}.

\begin{theorem}
\label{th:immdecomp} \cite[Proposition~2.3, Proposition~4.4]{RS}\ If
$Y$ is a totally nonnegative matrix, then
$\Imm_{\tau}^{\mathrm{TL}}(Y) \geq 0$.  For two subsets $I,J\subset
[n]$ of the same cardinality and $S=J\cup (\bar I)^\wedge$, we have
$$
\Delta_{I,J}(Y) \cdot \Delta_{\bar I, \bar J}(Y)
= \sum_{\tau \in \Theta(S)} \Imm_{\tau}^{\mathrm{TL}}(Y).
$$
\end{theorem}

Let now $I = i_1 < i_2 < \ldots < i_k$, $I' = i'_1 < i'_2 < \ldots < i'_k$, $J = j_1 < j_2 < \ldots < j_k$ and $J' = j'_1 < j'_2 < \ldots < j'_k$ be four $k$-tuples of positive integers such that $i_t \leq i'_t$ and $j_t \leq j'_t$ for each$1 \leq t \leq k$.

\begin{lemma}\label{L:inc}
For a totally nonnegative matrix $X$ we have
$$\frac{\Delta_{I,J}(X)}{\Delta_{I',J}(X)} \geq
\frac{\Delta_{I,J'}(X)}{\Delta_{I',J'}(X)}$$ as long as the
denominators are non-zero.
\end{lemma}

\begin{proof}
Let $Y$ be the $2k \times 2k$ submatrix of $X$ induced by the rows
in $I \cup I'$ and columns $J \cup J'$, where we repeat a row or a
column if it belongs to both of the sets (that is, $I \cup I'$ and
$J \cup J'$ are considered multisets).  We index rows and columns of
$Y$ again by $I \cup I'$ and $J \cup J'$. Whenever there is a
repeated column we consider the right one of the two to be in $J'$.
 Similarly whenever there is a repeated row we consider the bottom
one of the two to be in $I'$. Then $I' = \bar I$, $J' = \bar J$ and
we can apply the above theorem to the products
$\Delta_{I,J}(X)\Delta_{I',J'}(X)$ and $\Delta_{I',J}(X)
\Delta_{I,J'}(X)$.

The coloring of $4k$ points one obtains from from $I$ and $J'$ (that
is, $S = J'\cup (\bar I)^\wedge$) has the property that both in the
right and left columns there are more white vertices near the top
than black vertices. More precisely, the $t$-th white vertex is
above the $t$-th black vertex. This follows from the conditions $i_t
\leq i'_t$ and $j_t \leq j'_t$.  Its easy to see this property of
the coloring implies that all edges of a compatible matching have
either both endpoints on the left or both on the right. Indeed, if
there is an edge connecting the two sides then the non-crossing
condition implies that the vertices above its endpoints on either
side should have an equal number of black and white vertices.  This
contradicts ``more white vertices near the top''.

The coloring coming from $I$ and $J$ (that is, $S = J\cup (\bar
I)^\wedge$) is obtained by swapping black and white colors on the
left.  It follows that every compatible matching remains compatible.
Thus every Temperley-Lieb immanant occurring in the decomposition of
$\Delta_{I',J}(X) \Delta_{I,J'}(X)$ occurs also in the decomposition
of $\Delta_{I,J}(X)\Delta_{I',J'}(X)$.  Since immanants are
nonnegative, by Theorem \ref{th:immdecomp} we conclude that
$$\Delta_{I,J}(X)\Delta_{I',J'}(X)) - \Delta_{I',J}(X) \Delta_{I,J'}(X) \geq
0.$$
\end{proof}

Call a sequence $J^s = j^s_1 < j^s_2 < \ldots < j^s_k$ increasing if $j_t^s < j_t^{s+1}$ for any $t$ and $s$.

\begin{theorem}
For a totally positive matrix $X$, $I$ and $I'$ as above and any
increasing sequence $J^s$ the limit $$\lim_{s \to \infty}
\frac{\Delta_{I,J^s}(X)}{\Delta_{I',J^s}(X)}$$ exists and does not
depend on the choice of the sequence $J^s$.
\end{theorem}

\begin{proof}
The fact that the limit exists follows from Lemma \ref{L:inc}: the
ratio is non-increasing and remains nonnegative.  To see that it is
independent of $J^s$, assume there is another sequence $J'^s$. Then
for every element $J^s = j^s_1 < j^s_2 < \ldots < j^s_k$ there is an
element $J'^t = j'^t_1 < j'^t_2 < \ldots < j'^t_k$ such that $j^s_r
< j'^s_r$ for every $r$. This means that
$$\lim_{s \to \infty} \frac{\Delta_{I,J^s}(X)}{\Delta_{I',J^s}(X)} \leq
\lim_{s \to \infty}
\frac{\Delta_{I,J'^s}(X)}{\Delta_{I',J'^s}(X)}.$$ However in the
same way we obtain the backwards inequality.  Thus the two limits
are equal.
\end{proof}

\section{Some open problems}\label{sec:prob}
We collect here some questions and conjectures.

\noindent {\sc From Section \ref{sec:networks}.}
\begin{conjecture}
Corollary \ref{C:coeffs} holds for all $X \in \G_{\geq 0}$.
\end{conjecture}
\begin{question}
Can {\it every} TNN matrix be represented by a possibly
infinite, not necessarily acyclic cylindric ``network"?
\end{question}

\medskip \noindent {\sc From Section \ref{sec:upper}.}
\begin{problem}
Let $X$ be a TNN matrix.  Then every entry $\overline{x}_{ij}(t)$ of
$\overline{X}(t)$ is a totally positive function.  What is the
relationship between the poles and zeroes (see Theorem \ref{T:ET})
of different entries $\overline{x}_{ij}(t)$?
\end{problem}

\begin{problem}
Brenti \cite{Br2} has studied combinatorics of Polya frequency
sequences, as well as generalizations such as log-concave sequences.
Can his questions and results be generalized to $n>1$?
\end{problem}

\medskip \noindent {\sc From Section \ref{sec:comm}.}  The following
problem is inspired by \cite{BFZ}.

\begin{problem}
Let $w \in S_{\infty}$ be applied to an infinite curl via the maps
$\eta_i$, as in Corollary \ref{cor:sinf}. Describe the parameters of
the resulting product explicitly as rational functions of the
original parameters.
\end{problem}
\begin{example}
For $n=3$ applying $w = s_1 s_2 s_1$ to $N({\bf a})N({\bf b})N({\bf
c}) \dotsc$ we get $c'_1 = $

{\footnotesize{
$$\frac{c_3(a_1a_3b_1b_2+a_1b_2b_3c_1+a_1a_3b_2c_1+b_1c_1c_2c_3+b_1b_3c_1c_2+a_1a_3c_1c_2+a_1c_1c_2c_3+a_1b_3c_1c_2+c_1^2c_2c_3)}{a_3b_1b_2c_3+b_2b_3c_1c_3+a_3b_2c_1c_3+a_2a_3b_1c_3+a_2a_3b_1b_3+a_2a_3c_1c_3+c_1c_2c_3^2+b_3c_1c_2c_3+a_3c_1c_2c_3}$$}}
\end{example}

\medskip \noindent {\sc From Section \ref{sec:infprod}.}
\begin{problem}\label{prob:poles}
Suppose $X = \prod_{i=1}^\infty N(\a^{(i)})$ is an infinite product
of curls (or whirls).  Each entry of $\overline{X}(t)$ is a totally
positive function. What can we say about the poles and zeroes of
$\overline{x}_{i,j}(t)$, in terms of the parameters $\a^{(i)}$?
\end{problem}

\medskip
\noindent {\sc From Section \ref{sec:canon}.} The following problem
is non-trivial even when $Y = Y'$ is the identity matrix.
\begin{problem}[Multiplication of canonical forms] \label{prob:can}
Let $X = ZYW$ and $X' = Z'Y'W'$ be written in canonical form.  How
can one write $XX'$ in canonical form?
\end{problem}

\begin{problem}Repeat Problem \ref{prob:poles} for matrices in
canonical form.
\end{problem}

\medskip
\noindent {\sc From Section \ref{sec:semigroup}.} One can break
Problem \ref{prob:can} into smaller more specific problems.
\begin{question}[Commutation of infinite whirls with infinite curls]
Let $X \in RC$ (or $LC$) and $Y \in RW$ (or $LW$).  When is it
possible to write $XY$ as $Y'X'$, where $Y' \in RW$ (or $LW$) and
$X' \in RC$ (or $LC$)?
\end{question}
\begin{question}[Product of opposing whirls or curls]
Let $X \in RC$ (resp.~$RW$) and $Y \in LC$ (resp.~$LW$).  How does
one rewrite $XY$ in canonical form?
\end{question}
\begin{question}[Doubly-infinite whirls or curls]
How does one rewrite in canonical form a {\it {doubly infinite}}
whirl (resp.~curl), that is, a product of whirls (resp.~curls)
infinite in both directions?
\end{question}

\medskip \noindent {\sc From Section \ref{sec:fact}.}

\begin{problem}
Let $X = \prod_{i=1}^\infty N(\a^{(i)})$ be an infinite product of
curls, and suppose the given factorization of $X$ is obtained from
the ASW factorization by the action of $w \in S_\infty$ (via the
maps $\eta_i$ in Corollary \ref{cor:sinf}).  Find simple expressions
for $a_j^{(i)}$ in terms minor ratio limits.
\end{problem}

A special case of the following problem is discussed in \cite{LP}.

\begin{problem}
Give an interpretation of minor ratio limits when both column and
row indices are increasing sequences.  When do such limits exist?
\end{problem}



\end{document}

%% file: loo1.pstex_t
\begin{picture}(0,0)%
\includegraphics{loo1.pstex}%
\end{picture}%
\setlength{\unitlength}{1973sp}%
\begingroup\makeatletter\ifx\SetFigFont\undefined%
\gdef\SetFigFont#1#2#3#4#5{%
  \reset@font\fontsize{#1}{#2pt}%
  \fontfamily{#3}\fontseries{#4}\fontshape{#5}%
  \selectfont}%
\fi\endgroup%
\begin{picture}(3105,3510)(4236,-5736)
\put(5089,-5636){\makebox(0,0)[lb]{\smash{{\SetFigFont{6}{7.2}{\rmdefault}{\mddefault}{\updefault}{$v_1$}%
}}}}
\put(6664,-3274){\makebox(0,0)[lb]{\smash{{\SetFigFont{6}{7.2}{\rmdefault}{\mddefault}{\updefault}{$w_2$}%
}}}}
\put(6601,-5661){\makebox(0,0)[lb]{\smash{{\SetFigFont{6}{7.2}{\rmdefault}{\mddefault}{\updefault}{$v_2$}%
}}}}
\put(5176,-2961){\makebox(0,0)[lb]{\smash{{\SetFigFont{6}{7.2}{\rmdefault}{\mddefault}{\updefault}{$w_1$}%
}}}}
\put(4251,-4324){\makebox(0,0)[lb]{\smash{{\SetFigFont{6}{7.2}{\rmdefault}{\mddefault}{\updefault}{$\mathfrak h$}%
}}}}
\end{picture}%

%% file: loo2.pstex_t
\begin{picture}(0,0)%
\includegraphics{loo2.pstex}%
\end{picture}%
\setlength{\unitlength}{1973sp}%
\begingroup\makeatletter\ifx\SetFigFont\undefined%
\gdef\SetFigFont#1#2#3#4#5{%
  \reset@font\fontsize{#1}{#2pt}%
  \fontfamily{#3}\fontseries{#4}\fontshape{#5}%
  \selectfont}%
\fi\endgroup%
\begin{picture}(6889,2452)(3764,-4686)
\put(5189,-4611){\makebox(0,0)[lb]{\smash{{\SetFigFont{6}{7.2}{\rmdefault}{\mddefault}{\updefault}{$i$}%
}}}}
\put(6251,-4361){\makebox(0,0)[lb]{\smash{{\SetFigFont{6}{7.2}{\rmdefault}{\mddefault}{\updefault}{$i+1$}%
}}}}
\put(9220,-4606){\makebox(0,0)[lb]{\smash{{\SetFigFont{6}{7.2}{\rmdefault}{\mddefault}{\updefault}{$i$}%
}}}}
\put(10282,-4356){\makebox(0,0)[lb]{\smash{{\SetFigFont{6}{7.2}{\rmdefault}{\mddefault}{\updefault}{$i+1$}%
}}}}
\end{picture}%

%% file: loo8.pstex_t
\begin{picture}(0,0)%
\includegraphics{loo8.pstex}%
\end{picture}%
\setlength{\unitlength}{1973sp}%
\begingroup\makeatletter\ifx\SetFigFont\undefined%
\gdef\SetFigFont#1#2#3#4#5{%
  \reset@font\fontsize{#1}{#2pt}%
  \fontfamily{#3}\fontseries{#4}\fontshape{#5}%
  \selectfont}%
\fi\endgroup%
\begin{picture}(2861,2131)(3761,-4370)
\end{picture}%

%% file: loo6.pstex_t
\begin{picture}(0,0)%
\includegraphics{loo6.pstex}%
\end{picture}%
\setlength{\unitlength}{1579sp}%
\begingroup\makeatletter\ifx\SetFigFont\undefined%
\gdef\SetFigFont#1#2#3#4#5{%
  \reset@font\fontsize{#1}{#2pt}%
  \fontfamily{#3}\fontseries{#4}\fontshape{#5}%
  \selectfont}%
\fi\endgroup%
\begin{picture}(8569,5290)(4693,-6300)
\end{picture}%

%% file: loo3.pstex_t
\begin{picture}(0,0)%
\includegraphics{loo3.pstex}%
\end{picture}%
\setlength{\unitlength}{1973sp}%
\begingroup\makeatletter\ifx\SetFigFont\undefined%
\gdef\SetFigFont#1#2#3#4#5{%
  \reset@font\fontsize{#1}{#2pt}%
  \fontfamily{#3}\fontseries{#4}\fontshape{#5}%
  \selectfont}%
\fi\endgroup%
\begin{picture}(6892,2158)(3761,-4392)
\end{picture}%

%% file: loo7.pstex_t
\begin{picture}(0,0)%
\includegraphics{loo7.pstex}%
\end{picture}%
\setlength{\unitlength}{2368sp}%
\begingroup\makeatletter\ifx\SetFigFont\undefined%
\gdef\SetFigFont#1#2#3#4#5{%
  \reset@font\fontsize{#1}{#2pt}%
  \fontfamily{#3}\fontseries{#4}\fontshape{#5}%
  \selectfont}%
\fi\endgroup%
\begin{picture}(8749,1678)(3064,-3661)
\put(11039,-2855){\makebox(0,0)[lb]{\smash{{\SetFigFont{12}{14.4}{\rmdefault}{\mddefault}{\updefault}{$4$}%
}}}}
\put(10089,-3330){\makebox(0,0)[lb]{\smash{{\SetFigFont{12}{14.4}{\rmdefault}{\mddefault}{\updefault}{$4$}%
}}}}
\put(9151,-3330){\makebox(0,0)[lb]{\smash{{\SetFigFont{12}{14.4}{\rmdefault}{\mddefault}{\updefault}{$3$}%
}}}}
\put(9626,-3330){\makebox(0,0)[lb]{\smash{{\SetFigFont{12}{14.4}{\rmdefault}{\mddefault}{\updefault}{$3$}%
}}}}
\put(10576,-2855){\makebox(0,0)[lb]{\smash{{\SetFigFont{12}{14.4}{\rmdefault}{\mddefault}{\updefault}{$3$}%
}}}}
\put(11514,-2393){\makebox(0,0)[lb]{\smash{{\SetFigFont{12}{14.4}{\rmdefault}{\mddefault}{\updefault}{$3$}%
}}}}
\put(9614,-2855){\makebox(0,0)[lb]{\smash{{\SetFigFont{12}{14.4}{\rmdefault}{\mddefault}{\updefault}{$2$}%
}}}}
\put(10089,-2843){\makebox(0,0)[lb]{\smash{{\SetFigFont{12}{14.4}{\rmdefault}{\mddefault}{\updefault}{$2$}%
}}}}
\put(10089,-2393){\makebox(0,0)[lb]{\smash{{\SetFigFont{12}{14.4}{\rmdefault}{\mddefault}{\updefault}{$1$}%
}}}}
\put(10564,-2393){\makebox(0,0)[lb]{\smash{{\SetFigFont{12}{14.4}{\rmdefault}{\mddefault}{\updefault}{$1$}%
}}}}
\put(11039,-2393){\makebox(0,0)[lb]{\smash{{\SetFigFont{12}{14.4}{\rmdefault}{\mddefault}{\updefault}{$1$}%
}}}}
\put(9151,-2855){\makebox(0,0)[lb]{\smash{{\SetFigFont{12}{14.4}{\rmdefault}{\mddefault}{\updefault}{$1$}%
}}}}
\end{picture}%

%% file: loo5.pstex_t
\begin{picture}(0,0)%
\includegraphics{loo5.pstex}%
\end{picture}%
\setlength{\unitlength}{2368sp}%
\begingroup\makeatletter\ifx\SetFigFont\undefined%
\gdef\SetFigFont#1#2#3#4#5{%
  \reset@font\fontsize{#1}{#2pt}%
  \fontfamily{#3}\fontseries{#4}\fontshape{#5}%
  \selectfont}%
\fi\endgroup%
\begin{picture}(8450,2101)(2751,-4950)
\end{picture}%